\documentclass[10pt,a4paper,reqno]{amsart}
\usepackage{amsthm,amssymb,amstext,graphicx}
\usepackage{bbm}
\usepackage{amscd}
\usepackage{amssymb}
\usepackage{amsthm}
\usepackage{mathtools}
\usepackage{hyperref}
\usepackage[utf8]{inputenc}
\usepackage[T1]{fontenc}
\usepackage{lmodern}
\usepackage{xcolor}
\usepackage[bibliography= common]{apxproof}
\usepackage{tikz}
\usepackage{eso-pic}
\usepackage{ifthen}
\usepackage{comment}
\usepackage{bm}
\usepackage{enumitem}
\usepackage{hyperref}
\hypersetup{
  colorlinks=true,
  pdfauthor={© Sven Karbach},
}

\DeclareMathOperator\dom{dom}

\DeclareMathOperator*{\esssup}{ess\,sup}
\newcommand*\D{\mathop{}\!\mathrm{d}}
\newcommand*\E{\mathop{}\!\mathrm{e}}
\newcommand*\I{\mathop{}\!\mathrm{i}}
\def\e{{\mathrm{e}}}

\newtheorem{theorem}{Theorem}[section]
\newtheorem*{theorem*}{Theorem}
\newtheorem{lemma}{Lemma}[section]
\newtheorem{proposition}{Proposition}[section]

\theoremstyle{plain}

\theoremstyle{definition}
\newtheorem{definition}{Definition}[section]
\newtheorem*{definition*}{Definition}
\newtheoremstyle{example}
  {.3\baselineskip}
  {.3\baselineskip}
  {\normalsize}  
  {0pt}       
  {\bfseries} 
  {.}         
  {5pt plus 1pt minus 1pt} 
  {}          
\theoremstyle{example}

\newtheoremstyle{remark}
  {.2\baselineskip}
  {.2\baselineskip}
  {\normalfont}
  {}
  {\bfseries}
  {\ifx\thmnote\@gobble.\else\normalfont.\fi}
  {.5em}
  {}
\theoremstyle{remark}
\newtheorem{remark}{Remark}[section]

\newlist{theoremenum}{enumerate}{1}
\setlist[theoremenum]{label=\roman*), ref=\textup{\thetheorem~\roman*)}}

\newlist{lemenum}{enumerate}{1}
\setlist[lemenum]{label=\roman*), ref=\textup{\thelemma~\roman*)}}

\newtheorem*{assumption*}{\assumptionnumber}
\providecommand{\assumptionnumber}{}
\makeatletter

\newcommand{\opP}{\mathbf{P}}
\newcommand{\opQ}{\mathbf{Q}}

\renewcommand{\MR}{\mathbb{R}}
\newcommand{\MC}{\mathbb{C}}

\newcommand{\ME}{\mathbb{E}}

\newcommand{\MP}{\mathbb{P}}
\newcommand{\MQ}{\mathbb{Q}}

\newcommand{\cF}{\mathcal{F}}
\newcommand{\cA}{\mathcal{A}}
\newcommand{\cB}{\mathcal{E}}
\newcommand{\cC}{\mathcal{C}}
\newcommand{\cD}{\mathcal{D}}
\newcommand{\cE}{\mathcal{E}}

\newcommand{\cI}{\mathcal{I}}

\newcommand{\cK}{\mathcal{K}}
\newcommand{\cL}{\mathcal{L}}

\newcommand{\cO}{\mathcal{O}}
\newcommand{\cP}{\mathcal{P}}

\newcommand{\cG}{\mathcal{G}}

\newcommand{\cS}{\mathcal{S}}

\newcommand{\sE}{\mathsf{E}}

\newcommand{\df}{\coloneqq}
\newcommand{\one}{\mathbf{1}}
\DeclareMathAlphabet{\mymathbb}{U}{BOONDOX-ds}{m}{n}
\DeclareMathAlphabet\mathbfcal{OMS}{cmsy}{b}{n}
\newcommand{\zero}{\mymathbb{0}}
\newcommand{\interior}[1]{({\kern0pt#1})^{\textnormal{o}}}
\newcommand{\set}[1]{\left\{ #1\right\}}
\newcommand{\norm}[1]{\|#1\|}

\newcommand{\EX}[1]{\mathbb{E}\left[#1\right]}

\newcounter{Task}

\newcommand{\MRplus}{\MR^{+}}
\numberwithin{equation}{section}

\begin{document}
\sloppy
\title[Measure-Valued CARMA Processes]{Measure-Valued CARMA Processes} 

\author{Fred Espen Benth}
\address[Fred Espen Benth]{Department of Mathematics, University of Oslo, POBox 1053 Blindern, N-0316 Oslo, Norway}
\email{fredb@math.uio.no}

\author{Sven Karbach}
\address[Sven Karbach]{Korteweg-de Vries Institute for Mathematics and Informatics Institute, University of Amsterdam, Science Park 105-107, 1098 XG Amsterdam, Netherlands} 
\email{sven@karbach.org}

\author{Asma Khedher}
\address[Asma Khedher]{Korteweg-de Vries Institute for Mathematics, University of Amsterdam, Postbus 94248, NL–1090 GE Amsterdam, The Netherlands}
\email{A.Khedher@uva.nl}

\begin{abstract}
In this paper, we examine continuous-time autoregressive moving-average
(CARMA) processes on Banach spaces driven by L\'evy subordinators. We show their
existence and cone-invariance, investigate their first and second order moment
structure, and derive explicit conditions for their stationarity. Specifically, we
define a \emph{measure-valued CARMA} process as the analytically weak solution
of a linear state-space model in the Banach space of finite signed measures. By selecting suitable input, transition, and
output operators in the linear state-space model, we show that the resulting solution
possesses CARMA dynamics and remains in the cone of positive measures
defined on some spatial domain.  We also illustrate how positive measure-valued CARMA
processes can be used to model the dynamics of functionals of spatio-temporal
random fields and connect our framework to existing CARMA-type
models from the literature, highlighting its flexibility and broader
applicability.\newline{}   

\end{abstract}

\thanks{Fred Espen Benth and Asma Khedher are grateful for the financial support by the Research Foundation Flanders (FWO) under the grant FWO WOG W001021N}

\maketitle

\section{Introduction}
\label{sec:intro}
In this paper, we introduce and analyze a class of measure-valued, L\'evy-driven
continuous-time autoregressive moving-average (CARMA) processes, which we refer
to as \emph{measure-valued CARMA}. These processes will extend finite-dimensional
approaches to define ARMA-like time-series processes in a continuous-time
context to an infinite-dimensional functional setting.

More specifically, we generalize the concept of non-negative L\'evy-driven CARMA
processes to general separable Banach spaces by defining them as cone-invariant solutions to a particular class of continuous-time linear state-space models
driven by L\'evy subordinators. This framework encompasses existing CARMA
models, such as real-valued CARMA processes~\cite{Bro01, TC05},  their multivariate extensions~\cite{MS07, BK23}, and the Hilbert space formulation~\cite{BS18}. 
A particularly interesting case arises when the underlying Banach space is the
space of finite signed measures, and the cone is the set of all positive measures. In this scenario, measure-valued CARMA processes can be used to
capture the dynamics of (functionals of) spatio-temporal random fields. We
provide a comprehensive analysis of this measure-valued setting, demonstrating
how the short-memory and continuous-time attributes of CARMA processes adapt
naturally to infinite-dimensional, cone-valued systems.

\subsection{CARMA Processes and Linear-State Space Models}

On general state spaces, one can view CARMA processes as solutions to higher-order stochastic differential equations of the form
\begin{align}\label{eq:MCARMA-higher-order-SDE}
  \mathrm{D}^{p}X_{t} + \tilde{A}_{1} \mathrm{D}^{p-1} X_{t} 
  + \ldots + \tilde{A}_{p} X_{t} 
  = \tilde{C}_{0} \mathrm{D}^{q+1} L_{t} 
  + \tilde{C}_{1} \mathrm{D}^{q} L_{t} 
  + \ldots + \tilde{C}_{q} \mathrm{D} L_{t},
\end{align}
where $\mathrm{D} = \frac{\mathrm{d}}{\mathrm{d} t}$,
$\{\tilde{A}_{i}\}_{i=1}^{p}$ and $\{\tilde{C}_{j}\}_{j=0}^{q}$ are families
of linear operators for $p, q \in \mathbb{N}$, and $(L_{t})_{t \in\MR}$
denotes a two-sided L\'evy  process.\par{}  

To interpret a higher-order stochastic differential equation of the
form in~\eqref{eq:MCARMA-higher-order-SDE}, one draws inspiration from the analogous
case in ordinary differential equations. There, a higher-order linear equation is
reformulated as a higher dimensional first order system by introducing auxiliary
state variables. A similar approach applies to \eqref{eq:MCARMA-higher-order-SDE},
resulting in a linear state-space model with L\'evy input process, whose transition operator is given by
the companion block operator matrix of the characteristic polynomial of the
differential equation. In this way, one can also define CARMA processes in general
Banach spaces driven by L\'evy processes, provided the resulting linear
state-space model is well-posed in a stochastic strong sense, that is, one can define an Ornstein-Uhlenbeck (OU) process on a Cartesian product of the underlying Banach space.  For reference, linear
state-space models associated with CARMA processes in the multivariate setting
are discussed in~\cite{BS13} (with cone-valued extensions in~\cite{BK23}), and
their adaptation to Hilbert spaces appears in~\cite{BS18}. 

In general separable Banach spaces, the feasibility of this approach critically depends on both the properties of the Banach space and the characteristics of the driving L\'evy noise, since stochastic integration techniques are not universally available in all infinite-dimensional settings. In this paper, since our primary interest lies in positive measure-valued processes, we focus first on CARMA processes taking values in convex cones that are driven by L\'evy subordinators. We show that, under suitable conditions on the model parameters, solutions to the CARMA linear state-space equations exist and remain within the cone. From a modeling perspective, this enables the construction of CARMA processes that evolve within cones of general separable Banach spaces. But more importantly, it provides a coherent solution concept for CARMA models in this general infinite-dimensional setting.

In particular, we exploit the Pettis integral to define stochastic integrals with respect to the Poisson random measures associated with the driving L\'evy process. However, in the case of non-separable Banach spaces (such as the space of finite Borel measures equipped with the total variation norm) one must either restrict to finite-dimensional L\'evy subordinators or adopt a stochastically weak formulation of the CARMA stochastic differential equation. The latter approach is a common and natural alternative in the literature, which we briefly review and compare to our approach in the following section.

\subsection{Measure-valued Processes}
We consider a different approach than considered in large parts of the measure-valued process
literature~\cite{Eng10, Li11, Gil13}, where usually Markovian techniques are
used to establish so-called
\emph{superprocesses}. Indeed, in~\cite[Chapter 9]{Li11}, the author showed the
existence (in a stochastic weak sense) of a general class of processes taking
values in $M_+(E)$, the cone of finite positive measures on some topological
space $E$, called \emph{immigration superprocesses}. We also refer
to~\cite{cuchiero2022measurevalued}, where measure valued affine and polynomial
processes where investigated. In the analysis in~\cite{Li11}
and~\cite{cuchiero2022measurevalued}, the authors used the fact that $M_+(E)$
endowed with the weak topology is separable, and locally compact when $E$ is a locally compact Polish
space \cite{varadarajan1958weak}. This allows the use of the positive maximum principle to show the
existence of the associated martingale problem.  In our case, we consider the
cone $M_+(E)^p$ as a state space. If $E$ is compact, then the cone $M_+(E)^p$
equipped with the topology induced by the direct sum inner product $\langle \cdot, \cdot \rangle_p$ is a
locally compact Polish space as the finite product of locally compact Polish
spaces is itself locally compact and Polish. Hence following similar derivations as for example in
\cite{cuchiero2022measurevalued}, one could approach to prove existence of OU-processes by using the positive maximum principle.  However, since our model is an OU process driven by a measure-valued L\'evy subordinator, we
instead use the theory on integration with respect to Banach-valued L\'evy processes to prove the existence of an
analytically weak and stochastically strong solution of the linear state-space equations directly.\par

\subsection{Applications to Modeling Dynamics of Renewable Energy Markets}\label{sec:model-energy-forw}

CARMA processes are widely recognized for their tractability, interpretability
and flexible autocorrelation structure,  inherited by their discrete-time
ARMA versions. These features have led to their application across diverse
fields, including meteorology, engineering, and finance. In particular in the
renewable energy domain, CARMA processes have been used to model (deseasonalized)
weather and time-dependent climate variables, including wind speed~\cite{BB09,BB12}
and temperature~\cite{BB12,BT13} and solar irradiance~\cite{DRTL16,LGB}. Moreover,
in financial applications, CARMA models served as mean-reverting processes for
volatility~\cite{TT06, BL13, BK23} and power prices~\cite{BKM14, benth2020pricing}, underscoring
their versatility in the intersection of finance and weather modeling, which is
what we henceforth call the modeling of \emph{renewable energy markets}.

\subsubsection{Climate Data}
When analyzing climate data across broader geographic regions (for
instance the Netherlands) instead of a few fixed locations there is a need for CARMA models
that incorporate spatial dimensions, thus extending into the realm of
spatio-temporal random fields. Indeed, in practical energy modeling, aggregate
variables (e.g., average temperature over a region) drive market
dynamics. For instance, power prices in Southern Norway can be strongly
influenced by the regional average temperature, rather than precise temperature
measurements at individual locations. This is because renewable energy
production and consumption depend on weather variables but also require spatial
weighting based on factors such as population density (for heating demand) or
production capacity (for renewable power plant installations).

For example, let $C(t,x)$ be the \emph{capacity factor} for renewable power
production (wind or solar) at time $t$ and location $x\in\mathcal O$. The
capacity factor measures the production from a power plant with installed capacity 1MW, and is a dimensionless number taking values in the interval $[0,1]$. Integrating over a time period $[\tau_1,\tau_2]$ and an area where the installed capacity of plants at time $t$ is given by the function $\eta(t,x)$ for $x\in\mathcal O$, we get the total production $P(\tau_1,\tau_2;\mathcal O)$ (in MWh) as
\begin{align}
    P(\tau_1,\tau_2;\mathcal O)=\int_{\tau_1}^{\tau_2}\int_{\mathcal
  O}\eta(t,x)C(t,x)\D x \D t.
\end{align}
The capacity factor is depending on the wind speed (or solar irradiation) at
time $t$ in location $x$. Rather than modelling this field, we can view it as a
measure-valued process, $C(\D t,\D x)$ and model the production over $\cO$ as
\begin{align}\label{eq:production-measure}
P(\tau_1,\tau_2;\mathcal O)=\int_{\tau_1}^{\tau_2}\int_{\mathcal O}\eta(t,x)C(\D
  t,\D x).
\end{align}
The former approach to modeling spatio-temporal
dependencies involves function-valued processes, which represent variables as
functions over spatial domains that evolve dynamically over time through
infinite-dimensional stochastic differential equations. This framework has been
explored in various contexts; see, for example, \cite{benth2015derivatives,
 cox2024infinite, cox2022affine, cox2022infinite}. The latter approach that we
want to follow in this paper leverages measure-valued processes and was
introduced in \cite{cuchiero2022measurevalued}, where forward dynamics are
modeled using measure-valued affine processes. 

Motivated by the local CARMA dynamics of weather variables, we propose measure-valued CARMA processes as flexible and tractable models for
the dynamics of functionals of spatio-temporal variables. By extending CARMA models into the
spatio-temporal setting, one can capture both local dynamics at individual
points in space and the aggregate effects resulting from spatial integration. 
We propose in particular to model the capacity measure $C(\D t,\D x)$  as measure-valued CARMA, given that
locally CARMA models for the irradiation and wind speeds have been found through data analysis.

\subsubsection{Flow Forwards}
Another motivation comes from gas or electricity markets, where a distinguishing
feature is that \emph{flow forwards} deliver the underlying energy resource over a period, e.g., a day, week, month, quarter,
or year, instead of a fixed time, as with most other commodities~\cite{BK08}. As a consequence of the special structure of flow forwards, we may model the
price $F(t, \tau_1, \tau_2)$ of a flow forward with delivery over the time
interval $(\tau_1, \tau_2]$ at some time $0 \leq t \leq \tau_1$ prior to the
initial delivery date as a weighted
integral of \emph{instantaneous forward prices} $f(t,u)$ with instant delivery at a fixed
time $u$ with $\tau_1 < u \leq \tau_2$, as follows: 
\begin{align}\label{eq:flow-forward}
F(t, \tau_1, \tau_2) = \int_{\tau_1}^{\tau_2} w(u, \tau_1, \tau_2) f(t, u)\,\D u.
\end{align}
Here, the contract is financially settled at the end of the delivery period
$\tau_2$, and the weight function $w$ is given by the arithmetic average:
    \begin{align*}
    w(u, \tau_1, \tau_2) = \frac{1}{\tau_2 - \tau_1}.
    \end{align*}
    Note that the instantaneous forward price $f(t,u)$ is actually unobserved,
    and there exists no forward contract with fixed instantaneous delivery in
    the market. Therefore, we can again approach to model $f(t, \D u)$ as a
 measure-valued process such that the price of the flow forward price becomes
\begin{align}\label{eq:flow-forward-2}
F(t, \tau_1, \tau_2)=\int_{\tau_1}^{\tau_2}w(u, \tau_1, \tau_2) f(t,\D u).
\end{align}
Motivated by the CARMA electricity price model in~\cite{BKM14}, we propose to model $(f(t,\D u))_{t\geq 0}$ by
a measure-valued CARMA process in the spirit of~\cite{cuchiero2022measurevalued}, incorporating mean-reversion and a flexible
higher-order autoregressive structure observed in these markets.

\subsubsection{Power Purchase Agreements and Renewable Portfolios}

The models presented in equations~\eqref{eq:production-measure} and~\eqref{eq:flow-forward-2} facilitate the analysis of optimal allocation and decommissioning decisions for renewable energy plants, as well as the assessment of production volumes and revenues generated by renewable asset portfolios. Within this framework, the measure $C(\mathrm{d}t,\mathrm{d}x)$ captures the installed capacity across different locations and times. By multiplying this capacity with the stochastic spot power price at each location, the resulting quantity becomes a measure-valued stochastic process in both time and space.

Consider a Power Purchase Agreement (PPA) established at time $t \leq \tau_1$ for delivery during the period $[\tau_1, \tau_2]$. The price of this agreement can be represented as the conditional expectation (under an appropriate pricing measure) of future profits or losses resulting from the difference between spot prices and the contracted fixed price $K$. Specifically, the payoff for the off-taker at location $x$ is given by:
\begin{align*}
\text{Payoff}_{\text{off-taker}}(t;\tau_1,\tau_2,x) = \int_{\tau_1}^{\tau_2} V(u,x)\left(P(u,x)-K\right)\,\mathrm{d}u,
\end{align*}
where $V(u,x)$ denotes the realized power production (i.e., {\it volume}) at time $u$ and location $x$, and $P(u,x)$ denotes the corresponding spot power price.

Taking conditional expectations at time $t$, the value (or price) of the PPA for delivery between $\tau_1$ and $\tau_2$ becomes:
\begin{align*}
\mathrm{PPA}(t;\tau_1,\tau_2) &= \mathbb{E}_t\left[ \int_{\tau_1}^{\tau_2}\int_{\mathcal{X}} V(u,x)(P(u,x)-K)\,\mathrm{d}x\,\mathrm{d}u \right] \\
&= \int_{\tau_1}^{\tau_2}\int_{\mathcal{X}}\left[ g(t,u,x)\left( f(t,u,x)-K\right) + \Sigma_t(u,x) \right]\,\mathrm{d}x\,\mathrm{d}u.
\end{align*}
\clearpage
Here, $g$, $f$ and $\Sigma$ are as follows:
\begin{itemize}
\item $g(t,u,x) = \mathbb{E}_t[V(u,x)]$ is the forward expected production at time $t$ for delivery at time $u$ and location $x$.
\item $f(t,u,x)=\mathbb{E}_t[P(u,x)]$ is the forward price at time $t$ for delivery at time $u$ and location $x$.
\item $\Sigma_t(u,x)=\mathrm{Cov}_t[V(u,x),P(u,x)]$ is the conditional covariance between production and spot price, representing volumetric and price risks.
\end{itemize}

In a measure-valued framework, this complex expression simplifies elegantly to:
\begin{align*}
\mathrm{PPA}(t; \tau_1, \tau_2) = \int_{\tau_1}^{\tau_2}\int_{\mathcal{X}} \tilde{\eta}(t,u,x) X(t,\mathrm{d}x,\mathrm{d}u),
\end{align*}
with  $X(t,dx,du)=g(t,du,dx)f(t,du,dx)$ representing the combined measure of forward expected production and price. The integrand $\tilde{\eta}(t,u,x)$ incorporates both forward prices and the covariance term, capturing all relevant stochastic dynamics and spatial variation.

\subsection{Layout of the article}
The paper is structured as follows:
Section~\ref{sec:linear-state-space-models} examines continuous-time linear state
space models in separable Banach spaces driven by L\'evy subordinators. In
particular, we show the existence of weak solutions to linear state-space
equations on cones; introduce Banach-valued CARMA processes and study their
stationarity and distributional properties. In Section~\ref{sec:super-CARMA}, we
focus on the Banach space of finite signed measures defined on some topological
space, and introduce the measure-valued CARMA process. In
Section~\ref{sec:appication-expectation} we compute expectation functionals of measure-valued
CARMA processes motivated by their applications.

\section{Linear State-Space Models in Banach
Spaces}\label{sec:linear-state-space-models} 
In this section, we consider linear state-space models in general separable
Banach spaces driven by L\'evy noise. Since our focus lies on (non-negative)
measure-valued CARMA processes, we specialize the framework to linear
state-space models taking values in convex cones within Banach spaces. This
refinement imposes additional parameter constraints on the linear state-space
model and requires that the driving L\'evy process be \emph{non-decreasing}, but
it also provides a coherent solution concept in this general setting.  In
particular, we introduce a linear state-space model for CARMA processes on cones,
demonstrating both their existence and stationarity. 

\subsection{L\'evy Processes in Banach Spaces}
Throughout this paper, we adopt the following notational conventions. Let
$\mathbb{N}$ denote the set of natural numbers and let
$\mathbb{N}_0 := \mathbb{N} \cup \{0\}$ be the set of nonnegative integers. For
a complex number $z = a + \I b \in \mathbb{C}$, we write $\Re(z)$ and $\Im(z)$
for its real and imaginary parts, respectively. 

We let $(B,\|\cdot\|)$ be a separable Banach space with dual $B^*$, and use the
dual pairing $\langle f, x\rangle := f(x)$ for $f \in B^*$ and $x \in B$. We
denote by $\mathrm{Bor}(B)$ the Borel $\sigma$-algebra on $B$. Elements of $B$
are denoted by lowercase letters such as $x, y, z$, and elements of $B^*$ are
denoted by $f, g, h$. Here and throughout, we denote by $(\Omega,\mathcal F, \mathbb F,\mathbb P)$ a complete filtered probability space, where $\mathbb F=(\mathcal F)_{t\geq 0}$ is the filtration and $\mathbb P$ the probability measure. A $B$-valued L\'evy process $(L_t)_{t \geq 0}$ is a stochastic process with values in $B$ defined on the filtered probability space $(\Omega,\mathcal F, \mathbb F,\mathbb P)$ that satisfies:
\begin{itemize}
    \item[i)] $L_0 = 0$ almost surely,
    \item[ii)] $(L_t)_{t \geq 0}$ has independent and stationary increments,
    \item[iii)] $(L_t)_{t \geq 0}$ is stochastically continuous with respect to
        the norm $\|\cdot\|$, i.e., for every $\varepsilon > 0$,
$\mathbb{P}\bigl(\|L_t - L_s\| > \varepsilon\bigr) \to 0$ as $s \to t$,
    \item[iv)] $(L_t)_{t \geq 0}$ has right-continuous paths with left limits
        (càdlàg) almost surely, with respect to the norm $\|\cdot\|$.
\end{itemize} 
A set $A \in \mathrm{Bor}(B \setminus \{0\})$ is called
\emph{bounded-from-below} if $0$ does not lie in its closure under $\|\cdot\|$. 
For every $A \in \mathrm{Bor}(B \setminus \{0\})$ bounded from below and
$t > 0$, define
\begin{align*}
  N(t, A) \df \sum_{s \in [0, t]} \mathbf{1}_{A}\bigl(\Delta L_s\bigr),
\end{align*}
where $\Delta L_s := L_s - L_{s-}$. Since $(L_t)_{t \geq 0}$ has càdlàg paths,
there are only finitely many jumps of size larger than a positive constant in
any bounded-from-below set $A$. Hence, $(N(t, A))_{t \geq 0}$ is a Poisson
process, and we let $\ell(A) \df \mathbb{E}[N(1, A)]$ denote the \emph{L\'evy
measure}, which extends to a $\sigma$-finite measure on
$\mathrm{Bor}(B \setminus \{0\})$, finite on every bounded-from-below set, see also~\cite{RvG09} for additional details. 

The L\'evy process $(L_t)_{t \geq 0}$ is said to be \emph{integrable} if
$\mathbb{E}[\|L_t\|] < \infty$ for all $t \geq 0$, and \emph{square-integrable}
if $\mathbb{E}[\|L_t\|^2] < \infty$ for all $t \geq 0$. Note that the L\'evy
process $(L_t)_{t \geq 0}$ is square-integrable if and only if $$\int_{\set{z\in B\colon \norm{z}>1}}\norm{z}^{2}\ell(\D z)<\infty.$$

Set $D_0 := \{x \in B : 0 < \|x\| \leq 1\}$. From~\cite{GS75}, the
L\'evy-Khintchine representation for Banach-valued L\'evy processes states that, for every $f \in B^*$ and $t \geq 0$, the characteristic functional of $L_t$ is 
\begin{align}\label{eq:chracteristic-Levy}
  \mathbb{E}\big[\exp\big(\I \langle f, L_t\rangle\big)\big] 
  = \exp\Big( t\big(-\tfrac{1}{2}\langle f, Q f \rangle 
    + \I \langle f,\gamma\rangle + \psi(f)\big)\Big),
\end{align}
where
\begin{align*}
  \psi(f) \df \int_B \Big(\exp\big(\I \langle f,z\rangle\big) 
    - 1 - \I \langle f,z\rangle\, \mathbf{1}_{D_0}(z)\Big)\,\ell(\mathrm{d}z).
\end{align*}
In this representation, $(\gamma, Q, \ell)$ is the \emph{characteristic triplet}
of the L\'evy process, which can be interpreted as follows: $\gamma \in B$ is the drift vector
of the L\'evy process; $Q$ is the covariance operator of the continuous part of
the process, which is mapping from $B^{*}$ to $B$ and is non-negative and self-adjoint, i.e. $\langle x, Q x\rangle \geq 0$ for all $x\in B^*$ and $\langle y, Q x\rangle= \langle x,Q y\rangle$ for all $x,y\in B^{*}$; and $\ell$ is the L\'evy measure from before, defined on the Borel $\sigma$-algebra of $B \setminus \{0\}$ and is such that $\int_{D_0}|\langle f, z\rangle|^2\ell( \D z)<\infty$.
 
\subsection{L\'evy Processes on Cones in Banach Spaces}
A nonempty, closed, convex set $K \subseteq B$ is called a \emph{convex cone} if for any $\lambda \geq 0$ and $x \in K$, it holds that $\lambda x \in K$. A cone $K$ is said to be \emph{generating} if $B = K - K$, i.e. every $x \in B$ can be written as $x = y - z$, where $y,z \in K$. Moreover, we call the generating cone $K$ \emph{proper} if $x = 0$ whenever both $x \in K$ and $-x \in K$. Now, let $K$ denote a proper convex cone in $B$. We know from~\cite[Proposition 9]{ARA06}, that any $K$-increasing Lévy process in $B$, i.e. a L\'evy process $(L_t)_{t\geq 0}$ such that $L_t-L_s\in K$ $\MP$-a.s. for all $t\geq s$, assumes only values in $K$ and vice versa. We call a $K$-valued L\'evy process a \emph{subordinator}. 

Given a L\'evy measure $\ell$ on $\mathrm{Bor}(B\setminus \{0\})$, we shall say that an element $I_\ell \in B$ is  an \emph{$\ell$-Pettis centering} if
\begin{align}\label{eq:pettis-centering-1}
  \int_{D_0} |\langle f,z\rangle| \, \ell(\D z) < \infty \quad \text{for every } f \in B^*,  
\end{align}
and
\begin{align}\label{eq:pettis-centering-2}
 \langle f, I_\ell\rangle = \int_{D_0} \langle f,z\rangle \, \ell(\D z) \quad \text{for every } f \in B^*.   
\end{align}
We sometimes write $I_\ell = \int_{D_0} z \, \ell(\D z)$. Conditions sufficient
for the characteristic triplet $(\gamma,Q ,\ell)$ of a $B$-valued L\'evy process $(L_t)_{t\geq 0}$ to be a subordinator are given in~\cite{RA06}, the main result of which we recall in the following.
\begin{theorem}\label{thm:cone-subordinator}
Let $K$ be a proper convex cone of a separable Banach space $B$. Let $(L_t)_{t\geq 0}$ be a L\'evy process in $B$ with characteristic triplet $(\gamma, Q, \ell)$. Assume the following three conditions:  
\begin{theoremenum}
    \item $Q = 0$, 
    \item $\ell(B \setminus K) = 0$, i.e., $\ell$ is concentrated on $K$,
    \item\label{pettis} there exists an $\ell$-Pettis centering $I_\ell = \int_{D_0} z \, \ell(\D z)$ such that $\gamma_0 \df \gamma - I_\ell \in K$. 
\end{theoremenum}
Then the process $(L_t)_{t\geq 0}$ is a subordinator.
\end{theorem}

\noindent Observe that assumptions i)--iii) above give the particular
L\'evy-Khintchine representation (see ~\eqref{eq:chracteristic-Levy}):
\begin{align*}
 \EX{\exp\big(i \langle f, L_t\rangle\big)}= \exp\Big(t \big( i \langle f, \gamma_0\rangle+\int_{K\setminus\set{0}} \big( \E^{i \langle f, z\rangle} - 1 \big) \, \ell(\D z) \big) \Big),   
\end{align*}
since for all $f \in B^*$,
\begin{align*}
\langle f,\gamma_0\rangle = \langle f, \gamma\rangle - \int_{D_{0}\cap K} \langle f,z\rangle \, \ell(\D z).    
\end{align*}
We define the dual cone $K^*$ of $K$ by 
$$K^* = \{f \in B^*\colon \langle f, x\rangle \geq 0,\, \forall x \in K\}.$$

The Laplace transform of a subordinator $(L_t)_{t\geq 0}$ on a proper cone $K$ with Fourier transform
\begin{align*}
\mathbb{E}\left[ \exp\big(i \langle f, L_t\rangle\big) \right] = \exp\Big(t \big( i \langle f,\gamma_0\rangle+\int_{K\setminus\set{0}} \left( \E^{i \langle f, z\rangle} - 1 \right) \, \ell(\D z) \big) \Big),      
\end{align*}
is obtained for every $ f \in K^{*}$ by standard analytic continuation as
\begin{align}\label{eq:Laplace-Levy-Subordinator}
           \mathbb{E}\left[ \exp\big(\!-\!\langle f, L_t\rangle\big) \right] = \exp\Big(\!-\!t \big( \langle f,\gamma_0\rangle+\int_{K\setminus\set{0}} \big( 1 - \E^{-\langle f,z\rangle} \big) \, \ell(\D z)  \big)\Big).
\end{align}

\subsection{Linear State-Space Models in Banach
Spaces}\label{sec:linear-state-space}

Let $\mathcal{L}(B_1, B_2)$ denote the space of all bounded linear operators
acting from a Banach space $(B_1, \|\cdot\|_1)$ to another Banach space
$(B_2, \|\cdot\|_2)$. The operator norm is denoted by
$\|\cdot\|_{\mathcal{L}(B_1, B_2)}$, making $\mathcal{L}(B_1, B_2)$ itself a
Banach space. In the special case $B_1 = B_2 = B$, we write
$\mathcal{L}(B)$. Calligraphic letters, such as $\mathcal{A}$, denote operators
acting on the product space $B^p := B \times \ldots \times B$, where $p$ is a
positive integer. The product space $B^p$ is again a Banach space under the norm
\begin{align*}
  \|\mathbf{x}\|_p := \sum_{i=1}^{p} \|x^i\|, 
  \quad \text{for } \mathbf{x} = (x^1, \ldots, x^p) \in B^p.
\end{align*}
If $K$ is a convex cone in $B$, then $K^p$ is naturally a convex cone in $B^p$,
and the dual cone of $K^p$ is $\big(K^*\big)^p$. For an operator $\mathcal{A}$ on $B^p$, we denote by $(\mathcal{A}_{ij})_{1 \leq i,j \leq p}$ its block operator matrix representation. The adjoint of $\mathcal{A}$ is denoted by $\mathcal{A}^*$; similarly, if $A$ is an operator on $B$, then $A^*$ is its adjoint. The identity operator in $\mathcal{L}(B)$ is denoted by $\mathbb{I}$, and in $\mathcal{L}(B^p)$ by $\mathcal{I}_p$. For $\mathbf{g} \in (B^p)^*$, we write
$\langle \mathbf{g}, \mathbf{x}\rangle_p := \mathbf{g}(\mathbf{x})$. 

\begin{definition}[Linear State-Space Model in Banach Spaces]\label{def:linear-state-space-model-Banach}
  Let $p \in \mathbb{N}$, and let the tuple
  $(\mathcal{A}, \mathcal{E}, \mathcal{C}, L)$ consist of:
  \begin{enumerate}
    \item[i)] a \emph{state transition operator} 
    $\mathcal{A}\colon D(\mathcal{A}) \subset B^p \to B^p$, 
    \item[ii)] an \emph{input operator} 
    $\mathcal{E} \in \mathcal{L}(B, B^p)$,
    \item[iii)] an \emph{output operator}
    $\mathcal{C} \in \mathcal{L}(B^p, B)$,
    \item[iv)] a $B$-valued L\'evy process $L = (L_t)_{t \geq 0}$. 
  \end{enumerate}

  A \emph{continuous-time linear state-space model} on $B$, associated with
  $(\mathcal{A}, \mathcal{E}, \mathcal{C}, L)$, is given by the
  \textit{state-space equation}:
  \begin{align}\label{eq:state-space-X1}
    \mathrm{d} \mathbf{X}_t &= \mathcal{A} \mathbf{X}_t \,\mathrm{d}t 
      + \mathcal{E} \,\mathrm{d}L_t, \quad t \ge 0, \nonumber \\
    \mathbf{X}_0 &= \mathbf{x},
  \end{align}
  and an \emph{observation equation}:
  \begin{align}\label{eq:output-Y}
    Y_t = \mathcal{C} \mathbf{X}_t, \quad t \geq 0.\hspace{15mm}
  \end{align}
  A $B^p$-valued process $(\mathbf{X}_t)_{t \geq 0}$ satisfying
  \eqref{eq:state-space-X1} in a stochastic strong sense is called the
  \emph{state process}, and a $B$-valued process $(Y_t)_{t \geq 0}$ defined
  via~\eqref{eq:output-Y} is called the \emph{output process} of the model
  associated with $(\mathcal{A}, \mathcal{E}, \mathcal{C}, L)$.
\end{definition}

Of course, to ensure that the state and output processes are
well-defined, equation~\eqref{eq:state-space-X1} must be well-posed. In
 Banach spaces, the existence and uniqueness of solutions depend crucially on
 the properties of both the space and the driving noise. For separable Hilbert
 spaces, it is known that \eqref{eq:state-space-X1} admits a unique mild
solution under mild conditions on $L$, $\mathcal{A}$ and $\mathcal{E}$,
see~\cite{PZ07}. In more general UMD Banach spaces, existence and uniqueness results can be found in~\cite{App15, Rie15}.

In our setting, we are mainly interested in positive, i.e., cone-valued states, and therefore consider the state space to be a proper convex cone $K\subseteq B$, and the
product cone $K^p \subseteq B^p$ for the output and state processes,
respectively.
To guarantee that solutions to
equations~\eqref{eq:state-space-X1}--\eqref{eq:output-Y}, if they exist, remain in $K^p$ and
$K$, respectively, we must impose conditions on the tuple
$(\mathcal{A}, \mathcal{E}, \mathcal{C}, L)$ that ensure the cone
invariance of the solutions.

\begin{definition}[cf.\ \cite{Lem98}]
Let $B$ be a Banach space and $K \subseteq E$ a cone. A (possibly unbounded) linear operator $A\colon \dom(A) \subseteq B \to B$ is called \emph{quasi-monotone increasing} with respect to $K$ if for all $x, y \in \dom(A)$ it holds: $x \leq_K y \quad \text{and} \quad \langle f, x \rangle = \langle f, y \rangle$ for all $f \in K^*$ implies $\langle f, A(x) \rangle \leq \langle f, A(y) \rangle \quad \text{for all} \quad f \in K^*$, where $\leq_K$ denotes the partial order induced by $K$.
\end{definition}

Note that by \cite[Theorem 1]{Lem98} if $A$ is quasi-monotone and generates a strongly continuous operator semigroup $(S_t)_{t\geq 0}$ in $\cL(B)$, then $S_t(K) \subseteq K$ for all $t \geq 0$.

In the next proposition, we show that under suitable conditions on the cone $K$ and parameters $(\mathcal{A}, \mathcal{E}, \mathcal{C}, L)$ there exists a process $(\mathbf{X}_t)_{t\geq 0}$ with values in $K^p$ that, for any test function $\mathbf{g} \in D(\mathcal{A}^*) \subset (B^{p})^{*}$, satisfies the following weak integral equation:
\begin{align}\label{eq:weak-solution-state-space}
  \langle \mathbf{g}, \mathbf{X}_t \rangle_p &= \langle \mathbf{g}, \mathbf{X}_0 \rangle_p 
  + \int_0^t \langle \mathcal{A}^* \mathbf{g}, \mathbf{X}_s \rangle_p \,\D s
  + \int_0^t \langle \mathbf{g}, \mathcal{E} \,\D L_s \rangle_p.
\end{align}

\begin{proposition}\label{prop:OU-mild-sol}
Let $(B, \|\cdot\|)$ be a separable Banach space, and let $K \subseteq B$ be a
proper convex cone whose dual cone $K^*$ generates $B^*$. Suppose $(L_t)_{t \geq
0}$ is a L\'evy subordinator with characteristic triplet $(\gamma, 0, \ell)$,
where $\gamma \in B$ is the drift term, $0$ indicates the absence of a Gaussian component, and $\ell$ is a L\'evy measure concentrated on $K$ satisfying Theorem~\ref{pettis}. Denote by $N(\mathrm{d}s,\mathrm{d}z)$ the Poisson random measure for the jumps of $(L_t)_{t \geq 0}$, and define the \emph{compensated Poisson random measure} as
\begin{align*}
\tilde{N}(\D s, \D z) \df N(\D s, \D z) - \ell(\D z) \D s.     
\end{align*}
Further, assume that $\mathcal{A}$ is quasi-monotone and generates a strongly continuous semigroup $(\mathcal{S}_t)_{t \geq 0}$ on $B^p$, and that $\mathcal{E} \in \mathcal{L}(B, B^p)$ satisfies $\mathcal{E}(K) \subseteq K^p$. 

Then for every $\mathbf{x} \in K^p$, there exists a unique state process
$(\mathbf{X}_t)_{t \geq 0}$ that is the \emph{analytically weak solution} of the
linear state-space equation~\eqref{eq:state-space-X1} satisfying~\eqref{eq:weak-solution-state-space}. Moreover, this solution admits the variation-of-constant representation
\begin{align}\label{eq:variation-of-constant}
\mathbf{X}_t&=\cS_t\mathbf{x}+\int_0^t\cS_{t-s}\cB\gamma \D s+\int_0^t\int_{\set{z\in K\colon 0<\norm{z}\leq 1}}\cS_{t-s}\cB z\tilde{N}(\D s,\D z)\nonumber\\
    &\quad +\int_0^t\int_{\set{z\in K\colon \norm{z}> 1}} \cS_{t-s}\mathcal{E}z N(\D s,\D z), 
\end{align}
$\mathbb{P}$-almost surely, and remains in $K^p$ for all $t \geq 0$.
\end{proposition}
\begin{proof}
We follow the approach in \cite[Theorem 7.2]{RvG09} to show that the process $(\mathcal{S}_{t-s} \mathcal{E} z)_{s \leq t}$ is stochastically integrable with respect to $\tilde{N}(\D s, \D z)$.
According to \cite[Theorem 5.2]{RvG09}, this is equivalent to showing that $(\mathcal{S}_{t-s} \mathcal{E} z)_{s \leq t}$ is Pettis integrable with respect to the measure $\ell(\D z) \D s$ on $ K\times [0, t]$. 

To prove this, we need to show that for all $\mathbf{g} \in B^{p*}$, the following integral is finite:
\begin{align}\label{eq:integration-condition-1}
    \int_{0}^{t}\int_{\{ z \in K \colon 0<\|z\|\leq 1 \}} | \langle \mathbf{g}, \mathcal{S}_{t-s} \mathcal{E} z \rangle_p | \, \ell(\D z) \D s < \infty, \quad \forall t \geq 0,
\end{align}
and that there exists an element $Y_t \in B^p$ such that for all $\bm g \in (B^{p})^{*}$, it holds that 
\begin{align}\label{eq:integration-condition-2}
 \langle \mathbf{g}, Y_t \rangle =  \int_{0}^{t}\int_{\{ z \in K \colon 0<\|z\|\leq 1 \}} \langle \mathbf{g}, \mathcal{S}_{t-s} \mathcal{E} z \rangle \, \ell(\D z) \D s.
\end{align}
Write $\mathbf{g}=\mathbf{g}^{+}-\mathbf{g}^{-}$ with $\mathbf{g}^{+},\mathbf{g}^{-}\in (K^{p})^{*}$ and note that since $\langle \mathbf{g}, \mathcal{S}_{t-s} \mathcal{E} z \rangle_p = \langle \mathbf{g}^{+}\, \mathcal{S}_{t-s} \mathcal{E} z \rangle_p-\langle \mathbf{g}^{-}, \mathcal{S}_{t-s} \mathcal{E} z \rangle_p$ and $\mathcal{S}_{t-s} \mathcal{E} z\in K^{p}$ by assumption the integrability condition reduces to

$$\int_{0}^{t}\int_{\set{z\in K\colon 0<\norm{z}\leq 1}}\langle \mathbf{g}^{\pm},\mathcal S_s\cB z\rangle_p\, \ell(\D z)\D s=\int_{D_{0}\cap K}\langle A(t),z\rangle\,\ell(\D z)<\infty,$$
where $A^{\pm}(t)=\int_{0}^{t}\cB^{*}\mathcal S^{*}_s\mathbf{g}^{\pm}\D s\in B^{*}$ and the finiteness follows from the Pettis integrability of the L\'evy measure $\ell$. Similarly, define
$$
Y_t = \int_{0}^{t} \mathcal{S}_{t-s} \mathcal{E} I_\ell \, \D s,
$$
where
$$
I_\ell = \int_{\{ z \in K \colon 0<\|z\|\leq 1 \}} z \, \ell(\D z)
$$
is the Pettis centering of $\ell$. Then, by~\eqref{eq:pettis-centering-2} for all $\mathbf{v} \in B^{p*}$, we have
\begin{align*}
\langle \mathbf{g}, Y_t \rangle &= \int_{0}^{t} \langle \mathbf{g}, \mathcal{S}_{t-s} \mathcal{E} I_\ell \rangle \, \D s \\
&= \int_{0}^{t} \int_{\{ z \in K \colon 0<\|z\|\leq 1 \}} \langle \mathbf{g}, \mathcal{S}_{t-s} \mathcal{E} z \rangle \, \ell(\D z) \D s,
\end{align*}
which confirms \eqref{eq:integration-condition-2}. Therefore, by \cite[Theorem 7.2]{RvG09}, there exists an analytically weak solution $(\mathbf{X}_t)_{t \geq 0}$ satisfying \eqref{eq:weak-solution-state-space} and represented by the variation-of-constants formula \eqref{eq:variation-of-constant}. 
\end{proof}

Let $(P_t)_{t \ge 0}$ denote the transition semigroup associated with the state-space process $(\mathbf{X}_t)_{t \ge 0}$ from Proposition~\ref{prop:OU-mild-sol}, acting on a suitable class of functions $f \colon B^{p} \to \mathbb{R}$ by
\begin{align}\label{eq:transition-semigroup}
(P_t f)(\mathbf{x}) \df \mathbb{E}\left[ f(\mathbf{X}_t) \mid \mathbf{X}_0 = \mathbf{x} \right], \quad \mathbf{x} \in K^{p}. 
\end{align}
In particular, we are interested in evaluating $P_t$ for any $t\geq 0$ on exponential-type functions $f$ of the form 
$\mathrm{e}^{- \langle \mathbf{g}, \cdot \rangle_p}$, where $\mathbf{g} \in
(K^{p})^{*}$. The next proposition provides a closed-form expression for $P_t$
on these exponential functions and characterizes the infinitesimal generator of
the semigroup $(P_t)_{t\geq 0}$ on a suitable domain. 
\begin{proposition}\label{prop:transition-semigroup}
Let $(\mathbf{X}_t)_{t \ge 0}$ be the $K^p$-valued state-space process given by~\eqref{eq:variation-of-constant}. Then:
\begin{enumerate}
  \item[i)] For $\mathbf{g} \in (K^{p})^{*}$, the
transition semigroup $(P_t)_{t \ge 0}$ satisfies
\begin{align}\label{eq:Pt-exponential}
(P_t e^{-\langle \mathbf{g}, \cdot \rangle_p})(\mathbf{x}) &= \E^{-\langle \mathbf{g}, \mathcal{S}_t \mathbf{x} \rangle_p}+\exp\Big(- \int_0^t \big( \langle \mathbf{g}, \mathcal{S}_{t-u} \mathcal{E} \gamma_0 \rangle_{p}\D u\Big)\nonumber\\
&\quad \times \exp\Big(\int_{0}^{t}\int_{K\setminus\set{0}} \big( e^{\langle \mathbf{g}, \mathcal{S}_{t-u} \mathcal{E} z \rangle_p} - 1 \big) \, \ell( \D z ) \big)\D u \Big),
\end{align}
for all $\mathbf{x} \in K^p$, where $\gamma_0= \gamma - \int_{\{z \in K\colon 0<\|z\|\leq1\}} z \, \ell(\D z)$.
\item[ii)] Let $\mathbf{g}_1, \ldots, \mathbf{g}_n \in (K^p)^* \cap D(\mathcal{A}^*)$, and let $\phi \in C_0^2(\mathbb{R}^n)$, and set $\mathbf{u}(\mathbf{x}) := \big( \langle \mathbf{g}_1, \mathbf{x} \rangle_p, \ldots, \langle \mathbf{g}_n, \mathbf{x} \rangle_p \big)$ and $
\bm{\xi}(z) := \big( \langle \mathbf{g}_1, \mathcal{E} z \rangle_p, \ldots, \langle \mathbf{g}_n, \mathcal{E} z \rangle_p \big)$. Next, define the cylindrical function $f \colon B^p \to \mathbb{R} $ by
\begin{align*}
f(\mathbf{x}) = \phi\big( \langle \mathbf{g}_1, \mathbf{x} \rangle_p, \ldots, \langle \mathbf{g}_n, \mathbf{x} \rangle_p \big).
\end{align*}
Then $f \in D(\mathcal{G})$, and the generator $\mathcal{G}$ of the transition semigroup $(P_t)_{t \ge 0}$ is given by
\begin{align}\label{eq:generator}
(\mathcal{G} f)(\mathbf{x}) 
&= \sum_{i=1}^n \partial_i \phi(\mathbf{u}(\mathbf{x})) \cdot \langle \mathbf{g}_i, \mathcal{E} \gamma \rangle_p 
+ \sum_{i=1}^n \partial_i \phi(\mathbf{u}(\mathbf{x})) \cdot \langle \mathcal{A}^* \mathbf{g}_i, \mathbf{x} \rangle_p \\
&\quad + \int_K \left( \phi(\mathbf{u}(\mathbf{x}) + \bm{\xi}(z)) 
- \phi(\mathbf{u}(\mathbf{x})) 
- \sum_{i=1}^n \partial_i \phi(\mathbf{u}(\mathbf{x})) \cdot \xi_i(z) \cdot \mathbf{1}_{\|z\| \le 1} \right) \ell(\D z).\nonumber
\end{align}
\item[iii)] For $\mathbf{g} \in (K^{p})^{*}\cap D(\mathcal{A}^{*})$, the process $(\mathbf{X}_t)_{t \ge 0}$ is the unique solution to the martingale problem associated with the generator $\mathcal{G}$ in~\eqref{eq:generator}. Moreover, for each $i \in \{1,\ldots,p\}$, we have
\begin{align}\label{eq:martingale-solution-linear}
    \langle g^{(i)}, X_{t}^{i} \rangle &= \langle g^{(i)}, X_0^{i} \rangle +t\Big(\langle g^{(i)}, (\mathcal{E} \gamma )^{i} \rangle +\int_{\set{z\in K\colon \norm{z}> 1}} \langle g^{(i)}, (\mathcal{E} z )^{i} \rangle \, \ell(\D z)\Big)\nonumber\\
    &\quad+ \langle g^{(i)}, M_t^{i} \rangle+ \int_0^t \sum_{j=1}^p \langle \mathcal{A}_{ij}^* g^{(i)}, X_s^{j} \rangle\D s\,, 
\end{align}
where
\begin{align*}
 \langle g^{(i)}, M_t^{i} \rangle = \int_0^t \int_{K} \langle g^{(i)}, (\mathcal{E} z )^{i} \rangle \tilde{N}( \D s, \D z)   
\end{align*}
is a purely discontinuous martingale.
\end{enumerate}
\end{proposition}

\begin{proof}
The L\'evy-It\^o-decomposition of $L_t$ is given by
   \begin{align*}
        L_t=\gamma t+\int_0^t\int_{\set{z\in K\colon 0<\norm{z}\leq 1}}z\tilde{N}(\D s,\D z)+\int_0^t\int_{\set{z\in K\colon \norm{z}> 1}}z N(\D s,\D z),   
   \end{align*}
   and hence the variation-of-constant formula~\eqref{eq:variation-of-constant} can be compactly written as 
   \begin{align}\label{eq:variation-of-constant-compact}
     \mathbf{X}_t&=\cS_t\mathbf{x}+\int_{0}^{t}\cS_{t-s}\cB\D L_s,\quad t\geq 0.
   \end{align}
Now, let $\mathbf{g} \in (K^{p})^*$, then by~\eqref{eq:variation-of-constant-compact} 
\begin{align*}
\mathbb{E}\left[ e^{ -\langle \mathbf{g}, \mathbf{X}_t \rangle_p } \right] &= e^{ -\langle \mathbf{g}, \cS_t \mathbf{x} \rangle_p}\mathbb{E} \left[\exp\big(-  \int_0^t\langle \mathbf{g},\cS_{t-s}\cB \D L_s\rangle_p\big)\right],
\end{align*}
and similarly to~\cite[Corollary 4.29]{PZ07} it follows that
\begin{align*}
\mathbb{E}\left[\exp\big(-\int_0^t\langle \mathbf{g},\cS_{t-s}\cB \D L_s\rangle_p\big)\right]=\exp\Big(-\int_0^t \psi_{L}(\cB^{*}\cS^{*}_{t-s}\mathbf{g})\D s\Big),   
\end{align*}

where for every $f \in K^*$ we write $\psi_L(f)=\langle f,\gamma_0\rangle+\int_{K} \big( 1 - e^{-\langle f,z\rangle} \big) \, \ell(\D z)$
for the L\'evy characteristic exponent in~\eqref{eq:Laplace-Levy-Subordinator}, which yields the desired formula~\eqref{eq:Pt-exponential}.

The form of the generator $\cG$ follows by similar arguments as in the Hilbert space setting in~\cite[Theorem 5.4]{PZ07}. Since the solution $(\mathbf X_t)_{t\geq 0}$ is a stochastically strong solution, it is also the solution to the martingale problem of its generator. 

It is left to show that~\eqref{eq:martingale-solution-linear} holds. For $i=1,\ldots,p$,  consider $(0,\ldots,g^{(i)},\ldots,0)\in (K^{p})^*$ and insert into~\eqref{eq:weak-solution-state-space} to obtain:
\begin{align*}
     \langle g^{(i)}, X_t^{i} \rangle &= \langle g^{(i)}, X_0^{i} \rangle +t\langle
                                      g^{(i)}, (\cB\gamma)^{i} \rangle+\int_0^t
                                      \langle g^{(i)}, (\mathcal{A} \bm X_s)^{i}
                                      \rangle\D s \nonumber\\ &\quad+
                                                                \int_0^t\int_{\set{z
                                                                \in K\colon
                                                                0<\norm{z}\leq
                                                                1}} \langle
                                                                g^{(i)},
                                                                (\mathcal{E}z)^{i}\rangle
                                                                \tilde{N}(\D
                                                                s,\D
                                                                z)\nonumber\\
                                    &\quad+ \int_0^t\int_{\set{z\in K\colon
                                      \norm{z}> 1}} \langle g^{(i)},
                                      (\mathcal{E}z)^{i}\rangle N(\D s,\D z). 
\end{align*}
Using the block operator form $(\cA_{ij})_{i,j=1,\ldots,p}$ of $\cA$ we can write
\begin{align*}
 \langle g^{(i)}, (\mathcal{A} \mathbf{X}_s)^{i} \rangle = \sum_{j=1}^p \langle \mathcal{A}_{ij}^* g^{(i)}, X_s^{j} \rangle.   
\end{align*}
Thus, by defining $(M_t)_{t\geq 0}$ as the purely discontinuous martingale given by
\begin{align*}
 \langle g^{(i)}, M_t^{i} \rangle = \int_0^t \int_{K} \langle g^{(i)}, (\mathcal{E} z )^{i} \rangle \tilde{N}( \D s, \D z ),   
\end{align*}
we obtain the desired decomposition in~\eqref{eq:martingale-solution-linear}.
\end{proof}

\begin{remark}
  Proposition~\ref{prop:OU-mild-sol} offers one (under many possible) approaches for establishing the
existence of a mild (or weak) solution to the state-space
equation~\eqref{eq:state-space-X1}, which is particularly well-suited for our
cone-valued setting. If, for instance, the driving L\'evy process has finite
variation, it is also possible to construct a pathwise solution
to~\eqref{eq:state-space-X1}. In that scenario, the parameter constraints
required to preserve positivity (i.e., ensuring the state and output processes
remain in the cone) can be relaxed. In other words, if one can define a solution
more generally on the entire Banach space $B$, the need to maintain
cone-preserving parameter conditions can be dropped at this stage.
\end{remark}
\begin{remark}\label{rem:finite-dim-noise}
The separability of the Banach space plays a crucial role in the definition of Banach space–valued Lévy processes. Without separability, certain foundational properties, such as stochastic continuity, which relies on norm convergence, can break down. In particular, the norm difference
$\|L_t-L_s\|$
may fail to be measurable.
Furthermore, in order for the integrals on the right-hand side of \eqref{eq:variation-of-constant} to be well-defined, we require the strong measurability of the mapping
$s \mapsto \mathbf{1}_{\{s\leq t\}}\cS_{t-s} \cE z$.
This property is guaranteed in separable Banach spaces (see, e.g.~\cite[Theorem 1.1]{pettis1938integration}), but can fail in more general (non-separable) settings, thereby making the integral potentially ill-defined. 

However, we can still define Lévy subordinators in cones of (not necessarily separable) Banach spaces by considering a  finite-dimensional L\'evy subordinator and mapping it into the Banach space of interest. 

As an example, consider $\Phi\colon \mathbb{R}^d \to B$,  defined by 
\begin{equation}\label{eq:phi}
 \Phi(z) = \sum_{i=1}^d z^i b^i\,,
 \end{equation}
 where $b^1, \dots, b^d \in B$ are fixed elements of a (not necessarily separable) Banach space $B$. Let $(Z_k)_{k \in \mathbb{N}} \in \mathbb{R}^d_+$ be  i.i.d. random variables with distribution $\ell_0$ and $(N_t)_{t\geq 0}$ be a Poisson random variable in $\mathbb{R}$ with intensity $\varrho \in \mathbb{R}_+$ and independent of $(Z_k)_{k \in \mathbb{N}}$. We define a compound Poisson  process in $B$ as follows  
\begin{equation}\label{eq:levy-finite-noise}
L_t = \sum_{k=1}^{N(t)} \Phi(Z_k)\,, \qquad t\geq 0\,.
\end{equation}
 Its Lévy measure $\ell$ on $\mathrm{Bor}(B \setminus \{0\})$ is a pushforward of $\ell_0$ via the mapping $\Phi$, scaled by the intensity $\varrho$ 
$$
\ell(A) = \varrho \cdot \ell_0(\Phi^{-1}(A)),
$$
where $A \subset B \setminus \{0\}$ is a Borel set in $B$. In this case, it is easy to verify that $\int_{\{0<\|x\|\leq 1\}} |\langle f, x \rangle|\, \ell(\D x) < \infty$, for every $f \in B^*$. When we assume that $\Phi$ maps in a proper convex cone $K$ of the Banach space $B$, then $(L_t)_{t\geq 0}$ is a $K$-valued subordinator and for all $f \in K^*$, its Lévy-Khintchine representation is given by, 
$$
\mathbb{E} \left[ \exp(- \langle f, L_t \rangle) \right] = \exp \left( t \int_{\mathbb{R}^d_+ \setminus \{0\}} \left( 1-\e^{- \langle f, \Phi(z) \rangle}  \right) \varrho \, \ell_0(\D z) \right).
$$
Let $F\colon [0,T] \to \cL(B)$. Then the integral of $F$ with respect to $(L_t)_{t\geq 0}$ is pathwise defined by 
\[\int_0^t F(s) \, \D L_s = \sum_{\tau_k\leq t} F(\tau_k) \Delta L_{\tau_k} = \sum_{\tau_k\leq t} F(\tau_k) \Phi(Z_k)\,, \qquad t\geq 0\,,\]
where $\Delta L_{\tau_k}$ is the jump size at time $\tau_k$.
Therefore the claims of Proposition \ref{prop:OU-mild-sol} remain valid even without assuming the separability of the Banach space $B$, for the process $(\bm X_t)_{t\geq 0}$ defined in \eqref{eq:weak-solution-state-space}, when driven by a L\'evy subordinator with a finite-dimensional noise as specified in \eqref{eq:levy-finite-noise}. 
\end{remark}

\subsection{L\'evy-Driven Banach-Valued CARMA Processes}\label{sec:banach-valued-carma}
Let $K \subseteq B$ denote a proper convex cone and $(L_t)_{t \geq 0}$ a $K$-valued Lévy process in $B$
with characteristic triplet $(\gamma, 0, \ell)$ satisfying the conditions of
Theorem~\ref{thm:cone-subordinator}. Moreover, let us denote by $\zero$ the null operator in $\mathcal{L}(B)$. For $p \in \mathbb{N}$,
consider possibly unbounded linear operators 
\begin{align*}
  A_i \colon B \to B, 
  \quad i = 1, \dots, p,
  \end{align*}
each with domain $D(A_i) \subseteq B$ being dense. Furthermore, let 
\begin{align*}
  \mathbb{I} \colon B \to B
\end{align*}
be a (possibly unbounded) linear operator with dense domain $D(\mathbb{I})
\subseteq B$. We then define the \emph{state transition operator}
$\mathcal{A}_p \colon D(\mathcal{A}_p) \subset B^p \to B^p$ by the block operator matrix
\begin{align}\label{eq:Ap-Banach}
  \mathcal{A}_p \;\coloneqq\;
  \begin{bmatrix}
    \zero & \mathbb{I} & \zero & \cdots & \zero \\
    \zero & \zero & \mathbb{I} & \ddots & \vdots \\
    \vdots & \vdots & \ddots & \ddots & \zero \\
    \zero & \cdots & \cdots & \zero & \mathbb{I} \\
    A_p & A_{p-1} & \cdots & \cdots & A_1
  \end{bmatrix}.
\end{align}
In the context of linear state-space models, this operator is often referred to
as the \emph{companion block operator matrix} of the associated operator polynomial:
\begin{align}
  \label{eq:P}
  \opP(\lambda)\df
   \mathbb{I} \lambda^{p}-A_{1}\lambda^{p-1}-A_{2}\lambda^{p-2}-\ldots-A_{p},\quad \lambda\in\MC.
\end{align}

Furthermore, for some $\sE \in \mathcal{L}(B)$, we define the \emph{input operator}
$\mathcal{E}_p \in \mathcal{L}(B, B^p)$ by 
\begin{align}\label{eq:Bp-Banach}
  \mathcal{E}_p(x) 
  \;\coloneqq\; (\,0, \dots, 0,\, \sE x\,)^{\intercal}
  \;\in\; B^p,
  \quad \text{for all } x \in B.
\end{align}
Thus, $\mathcal{E}_p$ injects an element $x \in B$ transformed by the operator $E$ into the last coordinate of a vector in $B^p$.
 
Next, let $q \in \mathbb{N}_0$ with $q < p$, and let 
\begin{align*}
  C_i \in \cL(B), 
  \quad i = 0, \dots, p-1,
\end{align*}
be such that $C_i = \zero$ for $i = q+1, \dots, p-1$. We define the \emph{output
operator}  $ \mathcal{C}_q \colon B^p \to B$ by
\begin{align}\label{eq:Cq-Banach}
  \mathcal{C}_q \bigl(x^1, \ldots, x^p\bigr) 
  \;\coloneqq\; \sum_{i=1}^{q+1} C_{i-1}\, x^i.
\end{align}
In other words, $\mathcal{C}_q$ acts on the first $q+1$ coordinates of a vector
in $B^p$ via the operators $C_0, \dots, C_q$. Likewise, the operator polynomial associated with the output operator $\mathcal{C}_q$ is
\begin{align}
  \label{eq:Q}
  \opQ(\lambda)\df C_{0}+C_{1}\lambda+C_{2}\lambda^{2}+\ldots+C_{q}\lambda^{q},\quad \lambda\in\MC. 
\end{align}

\begin{definition}\label{def:carma-process1}
  Let $p, q \in \mathbb{N}_0$ with $q < p$. A \emph{pure-jump L\'evy-driven
CARMA$(p,q)$ process} in the Banach space $B$ is defined as the
output process $(Y_t)_{t \ge 0}$ of the continuous-time linear state-space model
(see Definition~\ref{def:linear-state-space-model-Banach}), driven by a L\'evy
subordinator $(L_t)_{t \ge 0}$ and governed by the parameter set $(\mathcal{A}_p, \mathcal{E}_p, \mathcal{C}_q, L)$. Concretely,
  \begin{align}\label{eq:CARMA-banach1}
  \begin{cases}
    \mathrm{d}\mathbf{X}_t = \mathcal{A}_p\,\mathbf{X}_t \,\mathrm{d}t + \mathcal{E}_p \,\mathrm{d}L_t, \\[6pt]
    Y_t = \mathcal{C}_q\,\mathbf{X}_t,
  \end{cases}
  \end{align}
  where $\mathbf{X}_0 = \mathbf{x} \in B^p$, and the operators $\mathcal{A}_p$, $\mathcal{E}_p$, and $\mathcal{C}_q$ are specified in \eqref{eq:Ap-Banach}, \eqref{eq:Bp-Banach}, and \eqref{eq:Cq-Banach}, respectively.
\end{definition}

Proposition~\ref{prop:OU-mild-sol} above  yields sufficient conditions for the
existence of a CARMA$(p,q)$  process in general separable Banach spaces. Indeed,
let $K$ be a proper convex cone with generating dual cone $K^{*}$. If  $\mathcal{A}_{p}$ generates a strongly continuous positive operator semigroup $(\mathcal{S}_t)_{t \geq 0}$ on $B^p$, i.e.,
$\mathcal{S}_t(K^p)\subseteq K^p$ for every $t \ge 0$, and if
$\mathcal{E}_{p} \in \mathcal{L}(B, B^p)$ as input operator satisfies
$\mathcal{E}_p(K)\subseteq K^p$, then the CARMA$(p,q)$  state-space process
$(\mathbf{X})_{t\geq 0}$ exists and assumes values in $K^{p}$ for any $\mathbf{X}_{0}=\mathbf{x}\in K^p$. In the next proposition, we shed some light on conditions for the tuple $(\mathcal{A}_p, \mathcal{E}_p, \mathcal{C}_q, L)$ to allow for cone-valued CARMA specifications.

\begin{proposition}\label{prop:pos-necessary-car}
  Let $(\mathcal{A}_p, \mathcal{E}_p, \mathcal{C}_q, L)$ be as in
\eqref{eq:Ap-Banach}, \eqref{eq:Bp-Banach}, and \eqref{eq:Cq-Banach} above and assume that $\sE(K)\subseteq K$. Moreover, let $(L_{t})_{t \geq 0}$ be a $K$-valued L\'evy process. Then the following statements hold:
  \begin{enumerate}
  \item[i)]\label{item:pos-necessary-car-1}
    If $C_{j}(K)\subseteq K$ for all $j \in \{0,\ldots,q\}$ and
    $\mathcal{A}_p$ is quasi-monotone with respect to $K^{p}$, then
    the associated CARMA$(p,q)$ process $(Y_t)_{t \ge 0}$ exists and  remains in $K$ for all initial values $\mathbf{X}_{0} \in K^{p}$.
  \item[ii)]\label{item:pos-necessary-car-2}
    Let $J\subseteq \set{1,\ldots,p}$. Define the product cone $K^{J,p} \in B^p$ as $(K^{J,p})_{i}=K$ for all $i\in J$ and $(K^{J,p})_{i}=B$ otherwise.  If $J\subseteq \set{1,\ldots,p}$ and the operator $\cS_{t}(K^{J,p})\subseteq  K^{J,p}$ for all $t\geq 0$ then $J = \{1, 2,\ldots, p\}$. 
  \item[iii)]\label{item:pos-necessary-car-3}
    Conversely, if $(Y_t)_{t \geq 0}$ is $K$-valued for all initial values
    $\mathbf{X}_{0} \in K^p$ and for every $K$-valued L\'evy process
    $(L_t)_{t \geq 0}$, then $C_j \in \pi(K)$ for all $j \in \{0,\ldots,q\}$.
    Moreover, if $K = C_j^{-1}(K)$ holds for all $j \in \{0,\ldots,q\}$, then
    $\mathcal{A}_p$ must be quasi-monotone with respect to $K^{p}$.
  \end{enumerate}
\end{proposition}
\begin{proof}
i) Suppose that $\mathcal{A}_p$ is quasi-monotone with respect to $K^p$. By definition, the semigroup $(\mathcal{S}_t)_{t\ge 0}$ generated by $\mathcal{A}_p$ satisfies 
\begin{align*}
\mathcal{S}_t\,\mathbf{x} \in K^p \quad \text{for all } \mathbf{x} \in K^p \text{ and } t \geq 0.
\end{align*}
Thus, for any initial condition $\mathbf{X}_0 \in K^p$, we have 
\begin{align*}
\mathcal{S}_{t-s}\,\mathbf{X}_0 \in K^p \quad \text{for all } t \ge s.
\end{align*}
If, in addition, $L_t$ is a $K$-increasing Lévy process, then 
\begin{align*}
\mathcal{E}_p\big(L_s - L_{s'}\big) \in K^p \quad \text{for all } s > s' \ge 0.
\end{align*}
It follows that the stochastic convolution
\begin{align*}
\int_{0}^{t}\mathcal{S}_{t-s}\,\mathcal{E}_p\,\mathrm{d}L_s 
\end{align*}
takes values in $K^p$ for all $t \ge 0$. Moreover, if $C_j \in \pi(K)$ for all $j \in \{0,\ldots,q\}$, then applying that the output operator $\mathcal{C}_q$ preserves positivity, it follows,
\begin{align*}
\mathcal{C}_q\big(\mathcal{S}_{t-s}\,\mathbf{X}_0\big) \in K \quad \text{and} \quad \mathcal{C}_q\Big(\int_{0}^{t}\mathcal{S}_{t-s}\,\mathcal{E}_p\,\mathrm{d}L_s\Big) \in K.
\end{align*}
By the variation-of-constant formula, we conclude that the output process $Y_t =
\mathcal{C}_q \mathbf{X}_t$ remains in $K$ for all $t \ge 0$. Part ii) and iii) follow from arguments analogous to those in \cite[Lemma 3.13 ii) and iii)]{BK23}.
\end{proof}

Part (ii) of Proposition~\ref{item:pos-necessary-car-2} demonstrates that the
specific form of the transition operator $\mathcal{A}_p$ ensures
quasi-monotonicity only with respect to the full product cone $K^p$, which is the
only proper product cone in this context. Part iii) tells us that, if the output
block operators $C_{j}$ for $j=0,1,\ldots,q$ are  invertible, then
quasi-monotonicity of $\cA_{p}$ is also necessary for the cone-invariance of the CARMA$(p,q)$ process. 

In the following proposition, we derive explicit formulas for the first- and second-moment structure of a CARMA$(p,q)$ process in Banach spaces.

\begin{proposition}\label{prop:moments-1}
Let $(Y_{t})_{t \geq 0}$ be a CARMA$(p, q)$, with $p > q$, process as in \eqref{eq:CARMA-banach1} driven by a pure-jump L\'evy process with a L\'evy measure satisfying 
$$\int_{\{z  \in K\colon \|z\|> 1\}} |\langle g, z\rangle| \,\ell(d z)< \infty\,,\qquad \forall g \in B^*.$$
Let $g \in B^*$ and $(\mathcal{S}_t^*)_{t \geq 0}$ be the adjoint semigroup generated by $\mathcal{A}_p^*$. Then, for all $t \geq 0$, it holds that
\begin{align*}
\mathbb{E}[\langle g, Y_t \rangle] &= \langle \mathcal{S}_t^* \mathcal{C}_q^* g, \bm x \rangle_p+\int_0^t \langle \mathcal{E}_p^*S^*_{t-s} \cC^*_qg, \gamma_0\rangle+\int_{K} \langle \mathcal{E}_p^* \mathcal{S}_{t-s}^* \mathcal{C}_q^* g, z \rangle \, \ell(\D z)\D s\,.
\end{align*}
Moreover, if the L\'evy measure satisfies
$$\int_{\{z  \in K\colon \|z\|> 1\}} |\langle g, z\rangle|^2 \,\ell(d z)< \infty\,,\qquad \forall g \in B^*,$$
it holds that 
\begin{align*}
\mathrm{Var}[\langle g, Y_t \rangle] &=
\int_{0}^{t}\int_{K} \langle \mathcal{E}_p^* \mathcal{S}_{t-s}^* \mathcal{C}_q^* g, z \rangle^2 \, \ell(\D z)\D s.
\end{align*}
\end{proposition}
\begin{proof}
Let $g \in B^*$ and consider the Laplace transform of $\langle g, Y_t \rangle$:
\begin{align*}
U(\theta) = \mathbb{E}\left[ \mathrm{e}^{ - \theta \langle g, Y_t \rangle } \right], \quad \theta \geq 0.    
\end{align*}
Using the expression for the Laplace transform of $(X_t)_{t\geq 0}$ (see Proposition~\ref{prop:transition-semigroup}), we have

\begin{align*}
U(\theta) &= \mathbb{E}\left[ \mathrm{e}^{ - \theta \langle \mathcal{C}^*_{q} g, \mathbf{X}_t \rangle_p } \right]\\
&=\exp\left\{ - \theta \langle \mathcal{S}_t^* \mathcal{C}_q^* g, \bm{x} \rangle_p - \int_0^t \theta \langle \mathcal{E}_p^*\cS^*_{t-s} \cC^*_qg, \gamma_0\rangle\, \D s\right. \\
&\qquad \quad +\left. \int_0^t \int_{K} \left( \mathrm{e}^{ - \theta \langle \mathcal{E}_p^* \mathcal{S}_{t - s}^* \mathcal{C}_q^* g, z \rangle } - 1 \right) \, \ell(\D z) \D s \right\}.
\end{align*}
Differentiating $U(\theta)$ with respect to $\theta$, we obtain
\begin{align*}
\frac{\D}{\D \theta} U(\theta) &= - U(\theta) \left[ \langle \mathcal{S}_t^* \mathcal{C}_q^* g, \bm{x} \rangle_p +\int_0^t \langle \mathcal{E}_p^*\
\cS^*_{t-s} \cC^*_qg, \gamma_0\rangle\, ds \right.\\
&\qquad \quad + \int_{0}^{t}\left.\int_{K} \langle \mathcal{E}_p^* \mathcal{S}_{t-s}^* \mathcal{C}_q^* g, z \rangle \, \mathrm{e}^{ - \theta \langle \mathcal{E}_p^* \mathcal{S}_{t-s}^* \mathcal{C}_q^* g, z \rangle } \, \ell(\D z)\D s \right].    
\end{align*}
Evaluating at $\theta = 0$ (since $U(0) = 1$), we get
\begin{align*}
 \frac{\D}{\D \theta} U(\theta) \bigg|_{\theta = 0} &= - \left[ \langle \mathcal{S}_t^* \mathcal{C}_q^* g, \bm{x}\rangle_p + \int_0^t \langle \mathcal{E}_p^*\cS^*_{t-s} \cC^*_qg, \gamma_0\rangle\, ds \right.\\
 &\qquad \quad \left.+\int_{0}^{t}\int_{K} \langle \mathcal{E}_p^* \mathcal{S}_{t-s}^{*} \mathcal{C}_q^* g, z \rangle \, \ell(\D z)\D
  s\right].  
\end{align*}
It follows, noting the assumptions of the proposition,
\begin{align*}
\mathbb{E}[\langle g, Y_t \rangle] &= - \frac{\D}{\D \theta} U(\theta) \bigg|_{\theta = 0} \\
&= \langle \mathcal{S}_t^* \mathcal{C}_q^* g, \bm x \rangle_p +\int_0^t \langle \mathcal{E}_p^*S^*_{t-s} \cC^*_qg, \gamma_0\rangle\, + \int_{K} \langle \mathcal{E}_p^* \mathcal{S}_{t-s}^* \mathcal{C}_q^* g, z \rangle \, \ell(\D z)\D s.
 \end{align*}

For the second moment, we compute the second derivative:
\begin{align*}
\frac{\D^2}{\D \theta^2} U(\theta) &= U(\theta) \left\{ \left[ \langle \mathcal{S}_t^* \mathcal{C}_q^* g, \bm x \rangle_p +\int_0^t \langle \mathcal{E}_p^*\cS^*_{t-s} \cC^*_qg, \gamma_0\rangle\, ds \right.\right.\\
&\qquad + \left.\int_{0}^{t}\int_{K} \langle \mathcal{E}_p^* \mathcal{S}_t^* \mathcal{C}_q^* g, z \rangle \, \mathrm{e}^{ - \theta \langle \mathcal{E}_p^* \mathcal{S}_{t-s}^* \mathcal{C}_q^* g, z \rangle } \, \ell(\D z)\D s \right]^2  \\
&\qquad \left. + \int_{0}^{t}\int_{K} \langle \mathcal{E}_p^* \mathcal{S}_{t-s}^* \mathcal{C}_q^* g, z \rangle^2 \, \mathrm{e}^{ - \theta \langle \mathcal{E}_p^* \mathcal{S}_{t-s}^* \mathcal{C}_q^* g, z \rangle } \, \ell(\D z)\D s \right\}.
\end{align*}
Evaluating at $\theta = 0$, we have
\begin{align*}
&\frac{\D^2}{\D \theta^2} U(\theta) \bigg|_{\theta = 0} \\
&\qquad = \left( \langle \mathcal{S}_t^* \mathcal{C}_q^* g, \bm x \rangle_p + \int_0^t \langle \mathcal{E}_p^*\cS^*_{t-s} \cC^*_qg, \gamma_0\rangle\, ds+\int_{K} \langle \mathcal{E}_p^* \mathcal{S}_{t-s}^* \mathcal{C}_q^* g, z \rangle \, \ell(\D z)\D s \right)^2 \\ 
&\qquad \quad + \int_{0}^{t}\int_{K} \langle \mathcal{E}_p^* \mathcal{S}_{t-s}^* \mathcal{C}_q^* g, z \rangle^2 \, \ell(\D z)\D s.    
\end{align*}
Observe that due to the assumptions of the proposition, it holds 
\begin{align*}
\mathbb{E}[\langle g, Y_t \rangle^2] = \frac{\D^2}{\D \theta^2} U(\theta) \bigg|_{\theta = 0}. 
\end{align*}
Hence the expression for the variance follows.
\end{proof}

\subsection{Stationary CARMA Processes on Cones in Banach Spaces}

In this section, we discuss the stationarity of Banach-valued state and output
processes with input parameters of CARMA type. To obtain stationary
representations we first extend Lévy processes to the full real line: Consider a $B$-valued Lévy process 
$(L_t^{(1)})_{t \geq 0}$ defined on the positive real line $\mathbb{R}_+$ and let $(L_t^{(2)})_{t \geq 0}$ be an independent and identically distributed $B$-valued Lévy
process. A \emph{two-sided $B$-valued Lévy process} $(L_t)_{t \in \mathbb{R}}$
can then be defined as
\begin{align}\label{eq:two-sided-Levy}
L_t := \one_{\mathbb{R}_+}(t) L_t^{(1)} - \one_{\mathbb{R}_-}(t) L_{-t-}^{(2)}, 
\end{align}
where $L_{t-}^{(2)} := \lim_{s \nearrow t} L_s^{(2)}$ for all $t \geq 0$, and
$\mathbb{R}_- := -\mathbb{R}_+ \setminus \{0\}$.  Note that the two sided L\'evy process
in~\eqref{eq:two-sided-Levy} is $K$-valued, whenever the characteristic triplet $(\gamma,
0,\ell)$ of $(L^{(1)}_t)_{t\geq 0}$ satisfies the conditions of Theorem~\ref{thm:cone-subordinator}.
 
For any linear operator $A$, we denote its spectrum by $\sigma(A)$. Moreover, the spectral bound of $A$, denoted by $\tau(A)$, is defined as 
\begin{align}\label{eq:spectral-bound}
\tau(A) \df \sup\{\Re(\lambda) \colon \lambda \in \sigma(A)\}.    
\end{align}
\begin{definition}[cf.\ \cite{HL98}]
Define the set of \emph{positive operators} $\pi(K) \subseteq \mathcal{L}(B)$
with respect to the cone $K$ by
\begin{align*}
\pi(K) \;\df\; \Big\{ A \in \mathcal{L}(B) : A(u) \;\geq_{K}\; 0 \;\;\text{for all}\;\; u \geq_{K} 0 \Big\}.
\end{align*}
We denote by `$\preceq$' the partial order on $\mathcal{L}(B)$ induced by
$\pi(K)$.
\end{definition}

From the variation-of-constants formula~\eqref{eq:variation-of-constant}, we observe that $Y_{t} \in K$ for every $t \in \mathbb{R}$, provided that the L\'evy process $(L_{t})_{t \in \mathbb{R}}$ is
$K$-increasing and that the function $G(s) \coloneqq \mathcal{C}_q \, \E^{s \mathcal{A}_p} \mathcal{B}_p$ satisfies $G(s) \in \pi(K)$ for every $s \geq 0$. To guarantee this property, we introduce the fundamental concept of complete monotonicity for operator-valued functions, see~\cite[Definition 5.4]{Are84}.

\begin{definition}
A function $f \colon \mathbb{R}_+ \to \mathcal{L}(B)$ is called \emph{completely monotone} with respect to $\pi(K)$ if $f$ is infinitely differentiable and
$$
(-1)^n f^{(n)}(\lambda) \succeq 0 \quad \text{for all} \quad \lambda > 0 \quad \text{and all} \quad n \in \mathbb{N}_0.
$$
\end{definition}

We have the following result on the stationarity of Banach-valued state and output processes with CARMA parameters:

\begin{proposition}\label{prop:stationary-CARMA}
Consider a CARMA$(p,q)$ linear state-space model defined by the parameter tuple
$(\mathcal{A}_p, \mathcal{E}_p, \mathcal{C}_q, L)$ as follows
\begin{align}\label{eq:CARMA-banach-stationary}
  \begin{cases}
    \mathrm{d}\mathbf{X}_t = \mathcal{A}_p\,\mathbf{X}_t \,\mathrm{d}t + \mathcal{E}_p \,\mathrm{d}L_t, \\[6pt]
    Y_t = \mathcal{C}_q\,\mathbf{X}_t,
  \end{cases}
\end{align}
where $\mathcal{A}_p$, $\mathcal{E}_p$, and $\mathcal{C}_q$ are as defined in \eqref{eq:Ap-Banach}, \eqref{eq:Bp-Banach}, and \eqref{eq:Cq-Banach}, respectively, and $(L_t)_{t\in\mathbb{R}}$ is a two-sided $K$-valued Lévy process with representation \eqref{eq:two-sided-Levy} satisfying
\begin{align*}
\mathbb{E}\!\left[\log\|L_1\|\right] < \infty.
\end{align*}
Moreover, assume that
\begin{align}\label{eq:stationary-solution-condition}
  \Bigl\{\lambda \in \mathbb{C} : 0 \in \sigma\bigl(\mathbf{P}(\lambda)\bigr)\Bigr\} \;\subseteq\; \Bigl\{ \lambda \in \mathbb{C} : \Re(\lambda) \neq 0 \Bigr\},
\end{align}
where $\mathbf{P}(\lambda)$ is the operator polynomial associated with $\mathcal{A}_p$ as
in~\eqref{eq:P}. Then there exists a unique 
stationary solution $(Y_{t})_{t\in\MR}$ to
equation~\eqref{eq:CARMA-banach-stationary} taking values in $K$ if and only if the function
\begin{align*}
\lambda \mapsto \mathbf{Q}(\lambda)\mathbf{P}(\lambda)^{-1}
\end{align*}
is completely monotone with respect to the partial order induced by $\pi(K)$,
where $\mathbf{Q}(\lambda)$ is the operator polynomial associated with $\mathcal{C}_q$
in~\eqref{eq:Q}. 
In that case, the CARMA$(p,q)$ process admits the stationary representation
\begin{align}\label{eq:stationary-solution}
Y_{t} = \int_{-\infty}^{\infty} \mathcal{K}(t - s) \,\mathrm{d}L_{s}, \quad t \in \mathbb{R},
\end{align}
where
\begin{align}\label{eq:stationary-kernel}
\mathcal{K}(t) = \frac{1}{2\pi} \int_{-\infty}^{\infty} \E^{\I \omega t}
  \mathbf{Q}(\I\omega)\mathbf{P}(\I\omega)^{-1} \,\mathrm{d}\omega, \quad t \in \mathbb{R}.
\end{align}
\end{proposition}
\begin{proof}
  Define the kernel $\cK\colon\MR\to\cL(B,B^p)$ by
  \begin{align}\label{eq:kernels-g}
   \cK(t)\df \lim_{R\to\infty}\frac{1}{2\pi \I}\int_{-\I R}^{\I R}\E^{\lambda t}(\lambda
    \cI_p-\cA_{p})^{-1}\cB_{p}\D \lambda,\quad t\in\mathbb{R}.
  \end{align}

Note that the term $(\lambda \cI_p-\cA_{p})^{-1}\cB_{p}$ can be computed explicitly by solving the following linear equation  
    \begin{align*}
      (\lambda \cI_p-\cA_{p})\mathbf{F}=\cB_p,
    \end{align*}
    for $\mathbf{F}\df(F_{1},F_{2},\ldots,F_{p})^{\intercal}\in\cL(B,B^p)$  with
$F_{i}\in \cL(B)$ for all $i=1,\ldots,p$. Indeed, by the specific form of the block
operators $\cA_p$ and $\cB_p$, we obtain the following explicit rational form:
$$\lambda\mapsto(\lambda
  \cI_p-\cA_{p})^{-1}\cB_{p}=(\mathbb{I},\lambda 
  \mathbb{I},\ldots,\lambda^{p-1}\mathbb{I})\mathbf{P}(\lambda)^{-1}.$$
    Therefore, we see that the complex integral in~\eqref{eq:kernels-g}, for every $ R\geq 0$, is well defined, since by assumption~\eqref{eq:stationary-solution-condition} there is no singularity of
  $$\lambda\mapsto(\lambda \cI_p-\cA_{p})^{-1}\cB_{p}=(\mathbb{I},\lambda\mathbb{I},\ldots,\lambda^{p-1}\mathbb{I})\mathbf{P}(\lambda)^{-1}$$ on the imaginary axis.  

Define a candidate stationary solution $(Y_t)_{t\in\mathbb{R}}$ by
\begin{align}\label{eq:Z-stationary-solution}
Y_{t} \coloneqq \int_{-\infty}^{t}\mathcal{C}_{q}\mathcal{K}(t-u)\,\mathrm{d}L_{u} - \int_{t}^{\infty}\mathcal{C}_{q}\mathcal{K}(t-u)\,\mathrm{d}L_{u}, \quad t\in\mathbb{R}.
\end{align}
We show that (i) these integrals are well-defined, and (ii) $(Y_{t})_{t\in\mathbb{R}}$ is the unique stationary solution to~\eqref{eq:CARMA-banach-stationary}.

Note first, that the integrals with respect to the L\'evy subordinator $(L_{t})_{t\in\MR}$ are again well defined, see the proof of
Proposition~\ref{prop:OU-mild-sol}, whenever $\cK(t)\in \pi(K)$ for all $t\in\MR$. 

Note, that from Bernstein’s theorem (see \cite[Theorem 5.5]{Are84}), the kernel $G(s)=\mathcal{C}_q\cS_{s}\mathcal{E}_p$ is positive with respect to $\pi(K)$ if and only if its Laplace transform $\varphi(\lambda)$ is completely monotone. But since, the Laplace transform, as we just computed, satisfies $\varphi(\lambda)=\mathbf{Q}(\lambda)\mathbf{P}(\lambda)^{-1}$, we see that $\mathcal{K}$ belongs to $\pi(K)$ and the stochastic integrals over $(-T,t]$ for any $T\in\MRplus$ are well-defined in the Pettis sense.

Next, note that there exist $\eta>0$ and $\delta>0$ such that for all $u\leq 0$ we have
  $\norm{\cK(-u)}_{\cL(B,B^{p})}\leq \eta \E^{-w(\cS) |u|}$
  and for all $u\geq 0$ we have $\norm{\cK(u)}_{\cL(B,B^{p})}\leq \eta
\E^{-w(\cS)u}$, see~\cite{EN00}, where $w(\cS)$ denotes the growth bound of $(\cS_{t})_{t\geq 0}$. This together with $\EX{\log(\norm{L_{1}})}<\infty$ implies
  the existence of the integrals  $\int_{-\infty}^{t}\cK(t-u)\D L_{u}$ and $\int_{t}^{\infty}\cK(t-u)\D L_{u}$, respectively, as limits of integrals over the
  intervals $(-T,t]$, resp. $[t,T)$, for $T\to\infty$,  see also~\cite{CM87}.

That $(Y_{t})_{t\in\MR}$ is a stationary solution to the CARMA$(p,q)$ state space equation then follows from the spectral representation of the semigroup $(\cS_{t})_{t\geq 0}$, see~\cite[Chapter 3, 5.15 Corollary]{EN00}, and representation~\eqref{eq:stationary-solution} follows.
\end{proof}

Finally, we derive the covariance structure of the Banach-valued state space process:
\begin{proposition}\label{prop:second-order-properties}
Let $p \in \mathbb{N}$ and $q \in \mathbb{N}_0$ with $q < p$, and let $\mathcal{A}_p$ be as in \eqref{eq:Ap-Banach} and $(\mathcal{S}_t)_{t \in \mathbb{R}}$ be the semigroup generated by $\mathcal{A}_p$. Moreover, let $\mathcal{E}_p$ and $\mathcal{C}_q$ be, respectively, as in \eqref{eq:Bp-Banach} and~\eqref{eq:Cq-Banach}. Let $(\mathbf{X}_t)_{t \in \mathbb{R}}$ be given by the Banach space-valued state-space model 
\[
\D \mathbf{X}_t = \mathcal{A}_p \mathbf{X}_t \, \D t + \mathcal{E}_p \, \D L_t, \quad Y_t = \mathcal{C}_q \mathbf{X}_t, \qquad t\in \mathbb{R}\,,
\]
with $\mathbf{X}_{0}=\mathbf{x}\in K^{p}$ and $(L_t)_{t\in \mathbb{R}}$ is a square-integrable Lévy process with covariance operator $\mathcal{Q}$ and values in $K$, the existence of which is given by Proposition~\ref{prop:OU-mild-sol}. Then, for all $s < t$, the conditional covariance operator of $Y_t$ given $\mathcal{F}_s$ is
\[
\operatorname{Var}[Y_t | \mathcal{F}_s] = \mathcal{C}_q \Sigma_{t,s} \mathcal{C}_q^*,
\]
where
\[
\Sigma_{t,s} = \int_s^t \mathcal{S}_{t - u} \mathcal{E}_p \mathcal{Q} \mathcal{E}_p^* \mathcal{S}_{t - u}^* \, \D u.
\]
Moreover, for $h \geq 0$, the autocovariance function is
\[
\operatorname{Cov}[Y_{t + h}, Y_t | \mathcal{F}_s] = \mathcal{C}_q \mathcal{S}_h \Sigma_{t,s} \mathcal{C}_q^*.
\]
If, in addition, $(\mathbf{X}_t)_{t \in \mathbb{R}}$ is stationary with representation
\begin{align}\label{eq:CARMA-stationary}
    \mathbf{X}_t = \int_{-\infty}^t \mathcal{S}_{t - s} \mathcal{E}_p \, \D L_s,
\end{align}
then
\[
\operatorname{Var}[Y_t] = \mathcal{C}_q \Sigma_\infty \mathcal{C}_q^*, \quad \forall t \in \mathbb{R},
\]
where
\[
\Sigma_\infty = \int_0^\infty \mathcal{S}_u \mathcal{E}_p \mathcal{Q} \mathcal{E}_p^* \mathcal{S}_u^* \, \D u,
\]
and the autocovariance function simplifies to $\operatorname{Cov}[Y_{t + h}, Y_t] = \mathcal{C}_q \mathcal{S}_h \Sigma_\infty \mathcal{C}_q^*$ for $h \geq 0.$
\end{proposition}

\begin{proof}
Since $\mathbf{X}_t$ is given by
\[
\mathbf{X}_t = \int_{-\infty}^t \mathcal{S}_{t - u} \mathcal{E}_p \, \D L_u,
\]
the conditional covariance operator given $\mathcal{F}_s$ (the $\sigma$-algebra up to time $s$) is
\[
\operatorname{Cov}[\mathbf{X}_t | \mathcal{F}_s] = \int_s^t \mathcal{S}_{t - u} \mathcal{E}_p \mathcal{Q} \mathcal{E}_p^* \mathcal{S}_{t - u}^* \, \D u = \Sigma_{t,s}.
\]

Then, the conditional variance of $Y_t = \mathcal{C}_q \mathbf{X}_t$ is
\[
\operatorname{Var}[Y_t | \mathcal{F}_s] = \mathcal{C}_q \operatorname{Cov}[\mathbf{X}_t | \mathcal{F}_s] \mathcal{C}_q^* = \mathcal{C}_q \Sigma_{t,s} \mathcal{C}_q^*.
\]

Similarly, the conditional covariance between $Y_{t + h}$ and $Y_t$ given $\mathcal{F}_s$ is
\begin{align*}
\operatorname{Cov}[Y_{t + h}, Y_t | \mathcal{F}_s] &= \operatorname{Cov}[\mathcal{C}_q \mathbf{X}_{t + h}, \mathcal{C}_q \mathbf{X}_t | \mathcal{F}_s] \\
&= \mathcal{C}_q \operatorname{Cov}[\mathbf{X}_{t + h}, \mathbf{X}_t | \mathcal{F}_s] \mathcal{C}_q^*.
\end{align*}

Since
\[
\mathbf{X}_{t + h} = \mathcal{S}_h \mathbf{X}_t + \int_t^{t + h} \mathcal{S}_{t + h - u} \mathcal{E}_p \, \D L_u,
\]
and the increments of $L_u$ are independent of $\mathcal{F}_t$, we have $\operatorname{Cov}[\mathbf{X}_{t + h}, \mathbf{X}_t | \mathcal{F}_s] = \mathcal{S}_h \operatorname{Cov}[\mathbf{X}_t | \mathcal{F}_s].$ Therefore,
$$
\operatorname{Cov}[Y_{t + h}, Y_t | \mathcal{F}_s] = \mathcal{C}_q \mathcal{S}_h \Sigma_{t,s} \mathcal{C}_q^*.
$$

If $(\mathbf{X}_t)_{t\in\mathbb R}$ is stationary, then the covariance operator is
$$
\Sigma_\infty = \int_0^\infty \mathcal{S}_u \mathcal{E}_p \mathcal{Q} \mathcal{E}_p^* \mathcal{S}_u^* \, \D u,
$$
provided the integral converges. The variance and autocovariance then follow accordingly.
\end{proof}

\section{L\'evy-Driven Measure-Valued CARMA Processes}\label{sec:super-CARMA}

In this section, we build on the results from Section~\ref{sec:linear-state-space-models}
to introduce pure-jump L\'evy-driven measure-valued CARMA$(p,q)$ processes, where $q < p$. These processes 
extend the CARMA framework to the setting of measure-valued state-space processes.

\subsection{The Space of Positve Finite Borel Measures}
We begin by introducing the necessary notation and mathematical preliminaries. Let $E$ be a Polish space with Borel $\sigma$-algebra $\operatorname{Bor}(E)$. Let $\lambda$ be a $\sigma$-finite measure on $(E, \operatorname{Bor}(E))$. Define $L^1(E)$ as the space of real-valued integrable measurable functions on $E$, i.e., 
$$L^1(E) = \left\{\alpha \colon E \to \mathbb{R} \mbox{ measurable}\, \Big\vert  \, \int_{E} \alpha(x) \lambda (\D x) <\infty\right \}\,.$$
Equipped with the norm $\|\alpha\|_{L^1} = \int_E |\alpha(x)| \, \lambda(\D x)$ the space $(L^1(E), \norm{\cdot}_{L^1})$ is a separable Banach space (see, e.g.~\cite[Lemma 23.19]{schilling2017measures}). Define $M_+(E)$ as the space of finite non-negative measures on $E$.  

Let $M(E) = M_+(E) - M_+(E)$ denote the space of all finite signed measures on $E$. Let $\Pi(E)$ denote the collection of all measurable partitions of $E$. For a signed measure $\nu$, the total variation measure is defined by
\begin{align*}
|\nu|(E) = \sup \left\{ \sum_i |\nu(E_i)| \colon E_i \in \pi(E), \, \pi(E) \in \Pi(E) \right\},
\end{align*}
which induces the total variation norm $\|\cdot\|_{M(E)}$ on $M(E)$. 

Equipped with this norm, $(M(E), \|\cdot\|_{M(E)})$ forms a {\it non-separable} Banach space \cite[Exercises 9a]{axler2020measure}. Let $C_0(E)$ denote the space of continuous and vanishing-at-infinity real-valued functions on $E$ equipped with the topology of uniform convergence. The supremum norm is denoted by $\|\cdot\|_{\infty}$. We define $C_{0,+}(E)$ as the subset of all positive elements in $C_0(E)$. Note that the space $C_0(E)$ is the dual of $M(E)$ and $C_{0,+}(E)$ the dual cone of $M_+(E)$.

Since separability of the Banach space is crucial for the results in Section~\ref{sec:linear-state-space-models} to hold in full generality (see Remark~\ref{rem:finite-dim-noise}), we restrict our attention to two classes of processes:
\begin{itemize}
    \item[(i)] measure-valued processes taking values in $M_+(E)$ but driven by \emph{finite-dimensional} L\'evy noise (cf.\ Remark~\ref{rem:finite-dim-noise});
    \item[(ii)] processes taking values in the subspace of $M(E)$ consisting of measures that are absolutely continuous with respect to $\lambda$.
\end{itemize}
The first case is fully covered by the theory developed in Section~\ref{sec:linear-state-space-models}, and below in Section~\ref{sec:examples} we give some examples of measure-valued CARMA processes driven by finite-dimensional noise. In the next section, we focus therefore on the second case. The corresponding subspace of $M(E)$ is isometrically isomorphic to $L^1(E)$ and thus forms a separable Banach space. This not only ensures that the assumptions of Section~\ref{sec:linear-state-space-models} are satisfied, but also enables more explicit representations of the model dynamics in terms of $L^1$-valued CARMA processes.

\subsection{An Absolutely Continuous Measure-Valued CARMA Process}\label{sec:super-OU}

Consider a measure $\mu$ of the form $\mu(\D x) = \alpha(x) \lambda(\D x)$, where $\alpha \in L^1(E)$. Then
\[
\|\mu\|_{M(E)} = \int_E |\alpha(x)|\,\lambda(\D x) = \|\alpha\|_{L^1(E)}\,.
\]
Thus, the subspace of absolutely continuous measures is isometrically isomorphic to $L^1(E)$, and its separability follows directly from the separability of $L^1(E)$. 

We denote by $L^1_+(E)$ the subset of $L^1(E)$ consisting of functions that are nonnegative almost everywhere. Clearly, $L^1_+(E)$ is a proper convex cone in $L^1(E)$. We define the nonzero part of this cone by $L^{1,\circ}_+(E) := L^1_+(E) \setminus \{0\}$.

Let $L^\infty(E)$ be the space of essentially bounded measurable functions, i.e.,
\[
L^\infty(E) = \left\{ g \colon E \to \mathbb{R} \,\middle|\, g \text{ measurable and } \esssup_{x \in E} |g(x)| < \infty \right\},
\]
equipped with the essential supremum norm $\|g\|_{L^\infty(E)} := \esssup_{x \in E} |g(x)|$. Then $(L^\infty(E), \|\cdot\|_{L^\infty})$ is a Banach space, and it forms the dual of $L^1(E)$ under the pairing
\[
\langle g, \alpha \rangle := \int_E g(x)\, \alpha(x)\, \lambda(\D x)\,.
\]

Under this duality, the dual cone of $L^1_+(E)$ is given by
\[
L^\infty_+(E) := \left\{ g \in L^\infty(E) \,\middle|\, g(x) \geq 0 \;\text{a.e. on } E \right\}.
\]

We emphasize that $L^1(E)$ is only one example of a separable Banach subspace of $M(E)$; we focus on it here to exemplify the theory. A full generalization would require a separate treatment of stochastic integration in non-separable Banach spaces.

Throughout this section, we use the notation $\alpha$ to refer to generic elements of $L^1_+(E)$ and $g$ to refer to elements of the dual cone $L^\infty_+(E)$. For elements in the $p$-fold Cartesian product of $L^1_+(E)$ or $L^\infty(E)$, we write $\bm\alpha = (\alpha^1, \ldots, \alpha^p) \in L^1_+(E)^p$ and $\bm g = (g^{(1)}, \ldots, g^{(p)}) \in L^\infty(E)^p$. The dual pairing on the product space is defined componentwise:
\[
\langle \bm g, \bm \alpha \rangle_p := \sum_{i=1}^p \langle g^{(i)}, \alpha^i \rangle = \sum_{i=1}^p \int_E g^{(i)}(x)\, \alpha^i(x)\, \lambda(\D x),
\]
and we define the $p$-norm by
\[
\|\bm \alpha\|_p := \sum_{i=1}^p \|\alpha^i\|_{L^1(E)}.
\]

\begin{definition}\label{def:super-OU}

Let $(L_t)_{t \geq 0}$ be an $L^1_+(E)$-valued Lévy process in $L^1(E)$ with characteristic triplet $(\gamma, 0, \ell)$, where $\gamma \in L^1_+(E)$ is the drift term and $\ell$ is the Lévy measure concentrated on $L^1_+(E)$ satisfying Theorem~\ref{pettis} with the cone $K$ being $L^1_+(E)$ and 
\begin{equation}\label{eq:int-measure}
\int_{L_+^{1, \circ}(E)} (\langle 1, \alpha\rangle \wedge 1 )\, \ell(\D \alpha)< \infty.
\end{equation}
Further, assume that $\mathcal{A} \colon D(\cA) \subset L^1_+(E)^p \rightarrow L^1_+(E)^p$ is the generator of a strongly continuous, quasi-positive operator semigroup $(\mathcal{S}_t)_{t \geq 0}$, and that the input operator $\mathcal{E} \in \mathcal{L}(L^1_+(E), L^1_+(E)^p)$ satisfies $\mathcal{E}(L^1_+(E)) \subseteq L^1_+(E)^p$. We define $(\mathbf{X}_t)_{t \geq 0}$ in $L^1_+(E)^p$, the existence of which is guaranteed by Proposition \ref{prop:OU-mild-sol}, to be the \emph{analytically weak solution} of 
\begin{align}\label{eq:state-space-X}
\D \mathbf{X}_t = \mathcal{A} \mathbf{X}_t \, \D t + \mathcal{E} \, \D L_t, \quad t \geq 0\,, \qquad \bm X_0=\bm \alpha \in L^1_+(E)^p\,.
\end{align}
We call the process $\bm{X}$ an $L^1(E)^p$-valued Ornstein-Uhlenbeck process.
\end{definition}

Let $\mathcal{A}$ be as in Definition~\ref{def:super-OU}, let $\mathcal{A}^*$ denote its adjoint, and let $(\mathcal{S}_t^*)_{t \geq 0} \colon L^\infty(E)^p \to L^\infty(E)^p$ denote the adjoint semigroup of $(\mathcal{S}_t)_{t \geq 0}$. 
Observe that the domain $D(\mathcal{A}^*)$ of $\mathcal{A}^*$ consists of all functions $\mathbf{f} \in L^\infty(E)^p$ for which the limit
$$
\lim_{t \downarrow 0} \frac{\mathcal{S}_t^* \mathbf{f} - \mathbf{f}}{t}
$$
exists in the norm
$$
\|\mathbf{f}\|_{\infty, p} \coloneqq \sum_{i=1}^p \|f^{(i)}\|_{\infty},
$$
where $\mathbf{f} = (f^{(1)}, \ldots, f^{(p)})^{\intercal}$.

Let $\cD_0$ be the class of functions $\xi\colon L^1(E)^p \rightarrow \mathbb{R}$
of the form 
$$\xi(\bm\alpha) = G(\langle \bm{g}_1, \bm{\alpha}\rangle_p, \ldots, \langle \bm{g}_n, \bm\alpha\rangle_p)\,,$$
where  $ G\in C_0^2(\mathbb{R}^n)$ and
$\bm{g}_1, \ldots, \bm{g}_n \in D(\cA^*)$. 

It holds that the Fr\'echet derivative of $\xi$ at $\bm \alpha \in L^1(E)^p$ is
$$\xi'(\bm{\alpha}) = \sum_{i=1}^n\partial_i G(\langle \bm{g}_1, \bm{\alpha}\rangle_p, \ldots, \langle
\bm{g}_n, \bm{\alpha}\rangle_p) \bm g_i \in L^\infty(E)^p,$$ where $\partial_iG$ denotes the $i^{th}$ partial derivative of $G$.  

For $\xi \in \cD_0$, define
\begin{align*}
\cG\xi(\bm \alpha) &= \sum_{i=1}^n \partial_i G(\langle \bm{g}_1, \bm{\alpha}\rangle_p, \ldots, \langle
\bm{g}_n, \bm{\alpha}\rangle_p) \\
&\qquad\qquad\times\left[\langle \cA^* \bm g_i, \bm \alpha\rangle_p + \langle \bm g_i, \mathcal{E}\gamma\rangle_p
- \int_{\{L_+^{1,\circ}(E)\colon \|\nu\|_{L^1(E)}\leq 1\}} \langle \bm g_i, \cB \sigma \rangle_p \, 
\ell(\D \sigma)\right] \\
&\qquad+ \int_{L_+^{1,\circ}(E)}\left(\xi(\bm \alpha+\mathcal{E}\sigma) -\xi(\bm\alpha)\right)\,\ell(\D \sigma)\,.
\end{align*}

In the following proposition we write $(\bm X_t)_{t\geq 0}$ as a solution to the martingale problem and give an It\^o type formula. The proof follows immediately from Proposition~\ref{prop:transition-semigroup}.
\begin{proposition}\label{prop:properties-X}
Let $(\bm{X}_t)_{t\geq 0}$ be as in Definition \ref{def:super-OU}. Then the following properties (all equivalent to each other) hold:
\begin{enumerate}
    \item[i)] Let $N(\D s, \D \alpha)$ denote the Poisson random measure associated with the jumps of $(L_t)_{t\geq 0}$, and let the compensated Poisson random measure be given by 
\begin{align*}
\tilde{N}(\D s, \D \alpha) = N(\D s, \D \nu) - \ell(\D \alpha) \D s.  
\end{align*}
    For any $\bm g\in D(\cA^*)$, $\bm X_0 =\bm \alpha \in L^1_+(E)^p$, it holds 
    \begin{align}\label{eq:weak}
    \langle g^{(i)}, X_t^i\rangle
    &= \langle g^{(i)}, \alpha^i\rangle +\langle g^{(i)}, M_t^i \rangle\nonumber\\
    &\qquad + \int_0^t \left(\sum_{j=1}^p\langle\cA_{ij}^* g^{(i)}, X_s^j\rangle + \langle g^{(i)}, (\cB\gamma)^i\rangle \right.\nonumber\\
    &\qquad \qquad  + \left.\int_{ \{\alpha \in L_+^{1,\circ}(E)\colon \|\alpha\|_{L^1}> 1\}} \langle g^{(i)}, (\mathcal{E}\alpha)^i\rangle \, \ell(d\alpha)\right)\, ds\,, 
    \end{align}
    where $\langle g^{(i)}, M_t^i\rangle = \int_0^t\int_{L_+^{1,\circ}(E)}\langle g^{(i)}, (\mathcal{E}\alpha)^{i} \rangle\, \tilde{N}(\D s, \D \alpha)$, $i=1,\ldots, p$,
    is a purely discontinuous local martingale 
    \item[ii)] $\forall \xi \in \cD_0$, $\bm X_0 =\bm \alpha \in L^1_+(E)^p$, we have 
    \begin{align*}
     \xi(\bm X_t) &= \xi(\bm \alpha) + \int _0^t \cG\xi(\bm X_s)\, \D s 
     \\
    &\qquad + \int_0^t\sum_{i=1}^n\partial_iG (\langle \bm g_i, \bm X_s\rangle_p, \ldots, \langle \bm g_n, \bm X_s\rangle_p)\langle \bm g_i, \D \bm M_s\rangle_p  \,,
    \end{align*}
    where $\langle \bm g_i, \D \bm M_s\rangle_p = \sum_{j=1}^p \int_{L_+^{1,\circ}(E)}\langle g_i^{(j)}, (\mathcal{E}\alpha)^j\rangle \tilde{N}(\D s, \D \alpha)$\,.
    \item[iii)] \label{item-5} For every $G \in C^2(\mathbb{R}^p)$ and $\bm g_1, \ldots, \bm g_n \in D(\cA^*)$, we have
    \begin{align}
    &G(\langle \bm g_1, \bm X_t\rangle_p, \ldots, \langle \bm g_n, \bm X_t\rangle_p )  \nonumber\\
    &\quad = G(\langle \bm g_1, \bm \alpha\rangle_p, \ldots, \langle \bm g_n, \bm \alpha\rangle_p) \nonumber\\
    &\qquad +\int_0^t\sum_{i=1}^n \partial_i G(\langle \bm{g}_1, \bm{X}_s\rangle_p, \ldots, \langle
\bm{g}_n, \bm{X_s}\rangle_p)\left[\langle \cA^* \bm g_i, \bm X_s\rangle_p + \langle \bm g_i, \mathcal{E}\gamma\rangle_p  \right. \nonumber\\
&\qquad \quad \left.- \int_{\{ \alpha \in L_+^{1,\circ}(E) \colon \|\alpha\|_{L^1}\leq1\}} \langle \bm g_i, \cB \nu \rangle_p \, 
\ell(\D \nu) \right]\,\D s\nonumber\\
    &\qquad +\int_0^t\int_{L_+^{1,\circ}(E)} G\left(\langle \bm g_1, \bm X_s +\mathcal{E}\alpha\rangle_p, \ldots,  \langle \bm g_n, \bm X_s +\mathcal{E}\alpha\rangle_p\right)\nonumber\\
    &\qquad \qquad - G\left(\langle \bm g_1, \bm X_s\rangle_p),\ldots, \langle \bm g_n, \bm X_s\rangle_p \right) \ell(\D \alpha) \, \D s
    \nonumber\\
    &\qquad + \int_0^t\sum_{i=1}^n\partial_iG (\langle \bm g_i, \bm X_s\rangle_p, \ldots, \langle \bm g_n, \bm X_s\rangle_p)\langle \bm g_i, \D \bm M_s\rangle_p\,.
    \end{align}
\end{enumerate}
\end{proposition}
\smallskip

 Now we are ready to introduce our pure-jump measure-valued CARMA($p,q$) process.

\begin{definition}\label{def:carma-process}
 Let $\mathcal{A}_p, \mathcal{E}_p, \mathcal{C}_q$ be respectively as in \eqref{eq:Ap-Banach}, \eqref{eq:Bp-Banach} and~\eqref{eq:Cq-Banach}, where we specify the Banach space $B$ to be the space of measures absolutely continuous with respect to a $\sigma$-finite measure $\lambda$. Assume $\cA_p$ to be quasi-monotone with respect to $L^1_+(E)^p$, $C_j \in \pi(L^1_+(E))$, for all $j\in \{0,\ldots, q\}$, and $\sE(L^1_+(E)) \subseteq L^1_+(E)$. Let $(L_t)_{t \geq 0}$ be a $L^1_+(E)$-valued Lévy process in $L^1(E)$ with characteristic triplet $(\gamma, 0, \ell)$, where $\gamma \in L^1_+(E)$ and $\ell$ is the Lévy measure concentrated on $L^1_+(E)$ satisfying \eqref{eq:int-measure} and Theorem~\ref{pettis} with the cone $K$ being  $L^1_+(E)$.
Let $(\bm X_t)_{t\geq 0}$ and $(Y_t)_{t\geq 0}$ be given by 
 \begin{align}\label{eq:CARMA-banach}
 \begin{cases}
 \D \mathbf{X}_t &= \mathcal{A}_p \mathbf{X}_t \, \D t + \mathcal{E}_p \, \D L_t,\\
Y_t &= \mathcal{C}_q \mathbf{X}_t,     
 \end{cases}
\end{align}
 where $\mathbf{X}_0 = \bm \alpha \in L^1_+(E)^p$. We call the process $(t,A)\mapsto \int_AY_t(x)\lambda(dx)$ 
a {\it a pure-jump measure-valued CARMA$(p,q)$} process with parameter set $(\mathcal{A}_p, \mathcal{E}_p, \mathcal{C}_q, L)$.
\end{definition}

In the following proposition we derive some properties of the process $(Y_t)_{t\geq 0}$. The proof follows from Propositions~\ref{prop:moments-1} and~\ref{prop:properties-X}.

\begin{proposition}
Let $(Y_{t})_{t\geq 0}$ be as described in Definition \ref{def:carma-process} and such that $Y_0 =\cC_q \bm \alpha$, for $\bm \alpha \in L^1_+(E)^p$. 

\begin{enumerate}
\item[(i)] For all $g$, such that $\cC_q^*g \in D(\cA^*_p)$, it holds for $t\geq 0$,
\begin{align*}
    \langle g, Y_t\rangle   
    &= \langle C^*_q g, \bm \alpha\rangle_p + \int_0^t\left(\langle \cA_p^*\cC^*_qg, \bm X_s\rangle_p + \langle (\cC^*_q g)^{(p)}, \sE\gamma\rangle\right)\, \D s \\
    &\qquad +\int_0^t \int_{\{\alpha \in L_+^{1,\circ}(E) \colon\| \alpha\|_{L^1}> 1\}}  \langle (\cC^*_q g)^{(p)}, \sE\alpha\rangle \,\ell(\D  \alpha)\, \D s\, \\
    &\qquad + \int_0^t \int_{L^{1,\circ}_+(E)}  \langle (\cC^*_q g)^{(p)}, \sE\alpha \rangle \tilde{N}(\D s, \D \alpha)\,.
\end{align*}
\item[(ii)] For all $g \in C_0(E)$, $t\geq 0$,
\begin{align}\label{eq:semigroup-R}
    \mathbb{E}[\mathrm{e}^{-\langle g,  Y_t\rangle}] &= \exp\left\{-\langle \cS_t^*\mathcal{C}_q g, \bm \alpha \rangle_p - \int_0^t\langle (\cS^*_{t-s}\cC_q^*g)^{(p)}, \sE\gamma\rangle\, \D s\, \right.\nonumber\\
    &\qquad + \left. \int_0^t \int_{L_+^{1,\circ}(E)} \left(\e^{-\langle (\cS_{t-s}^* \cC^*_qg)^{(p)}, \sE\alpha\rangle}-1\right)\,\ell(\D\alpha)\, \D s\right\}
\end{align}
\item[(iii)] Assume the L\'evy measure satisfies
$$\int_{\{\|\alpha\|_{L^1}\geq 1\}} |\langle 1, \alpha \rangle| \,\ell(d \alpha)< \infty\,.$$

Then, for $g \in C_0(E)$, $t \geq 0$, 
\begin{align*}
\mathbb{E}[\langle g, Y_t \rangle] &= \langle \mathcal{S}_t^* \mathcal{C}_q^* g, \bm \alpha \rangle_p \!+\!\int_0^t \langle \mathcal{E}^*_p\cS^*_{t-s} \cC^*_qg, \gamma_0\rangle\, \D s\\
&\quad +\int_{0}^{t}\int_{L_+^{1,\circ}(E)} \langle (\mathcal{S}_{t-s}^* \mathcal{C}_q^* g)^{(p)}, \sE \alpha \rangle \, \ell(\D \alpha)\D s.
\end{align*}
Moreover, assume the L\'evy measure satisfies
$$\int_{\{\|\alpha\|_{L^1}\geq 1\}} |\langle 1, \alpha\rangle|^2 \,\ell(\D \alpha)< \infty\,.$$
Then, for $g \in C_0(E)$, $t \geq 0$, 
\begin{align*}
\mathrm{Var}[\langle g, Y_t \rangle] &= \int_{L_+^{1,\circ}(E)} \langle (\mathcal{S}_{t-s}^* \mathcal{C}_q^* g)^{(p)}, \sE\alpha \rangle^2 \, \ell(\D \alpha).
\end{align*}
\end{enumerate}
\end{proposition}

Notice that the stationarity of the measure-valued CARMA processes introduced in Definition \ref{def:carma-process} follows from Proposition \ref{prop:stationary-CARMA}. 

\subsection{Examples of Parameter Sets $(\cA_p, \cB_p, \cC_q, L)$.}

In the following sections, we give some examples for the model parameters $(\cA_p, \cB_p, \cC_q, L)$.

\subsubsection{Convolution operator} Let $f \in L_{+}^{1}(E)$.   
We define the \emph{convolution operator} $T \colon L^1(E) \rightarrow L^1(E)$ by
\[
(T\alpha)(x) = (\alpha * f)(x) = \int_E f(x - y)\,\alpha(y)\, \lambda(\mathrm{d}y), \quad \text{for } \alpha \in L^1(E),
\]
where we assume $E$ has a
group structure where addition is defined. The space $L^1(E)$, equipped with the convolution product, forms a \emph{Banach algebra}. Moreover, by \emph{Young's inequality}, it holds that:
\[
\|\alpha * f\|_{L^1} \leq \|\alpha\|_{L^1} \|f\|_{L^1}.
\]
Hence, $T$ is a \emph{bounded linear operator} on $L^1(E)$ and maps non-negative functions to non-negative functions, i.e., $T(L_+^1) \subset L_+^1$, so it is \emph{positive}.

We now construct a \emph{strongly continuous semigroup} $(T_t)_{t \geq 0}$ of convolution type using the exponential formula (See \cite{EN00} for general properties of semigroups for bounded operators)
$$
T_t \alpha = \alpha * f_t, \qquad t \geq 0,
$$
where
$$
f_t = \sum_{n=0}^\infty \frac{t^n}{n!} f^{*n}, \quad f^{*0} := \delta_0, \quad f^{*n} := f * f^{*(n-1)} \text{ for } n \geq 1.
$$
Note that the identity element $T_0 = \mathrm{Id}$ is understood in the \emph{weak (distributional) sense}, since $\delta_0 \notin L^1(E)$. More precisely, we have
$$
\lim_{t \to 0^+} T_t \alpha = \alpha \quad \text{in } L^1(E),
$$
for all $\alpha \in L^1(E)$.
The family $(T_t)_{t \geq 0}$ then defines a \emph{strongly continuous semigroup} on $L^1(E)$, and it satisfies the abstract Cauchy problem
$$
\frac{\mathrm{d}}{\mathrm{d}t} T_t \alpha = T T_t \alpha = T_t T \alpha, \quad T_0 = \text{Id}.
$$
An example of the operators $A_i\colon L^1_+(E) \rightarrow L^1_+(E)$, $i=1,\ldots, p$, defining the matrix operator $\cA_p$ in \eqref{eq:CARMA-banach} would be the convolution operator. Similarly, the operators $\mathbf{C}_i \colon L^1_+(E) \rightarrow L^1_+(E)$, $i=0, \ldots p-1$, can be chosen as convolution operators. 

 \subsubsection{A L\'evy subordinator in $L^{1}_{+}(\mathbb{R}^d)$}\label{example2} Let $N(\D s, \D x, \D u)$ be a Poisson random measure on $\text{Bor}([0,\infty))\otimes \text{Bor}(\mathbb{R}^d)\otimes\text{Bor}((0,\infty))$ with intensity $\D s \otimes \lambda(\D x) \otimes \pi(\D u)$, where 
$\lambda(\D x)$ is a finite positive measure on $\text{Bor}(\mathbb{\mathbb{R}}^d)$ representing the spatial distribution that dictates where jumps can occur in space (i.e., over $\mathbb{R}^d$),
and $\pi(\D u)$ is a L\'evy measure on $\text{Bor}((0, \infty))$, describing the jump size distribution at each spatial location. Assume that 
$$\int_0^\infty  (u \wedge 1)  \, \pi(\D u) <\infty\,.$$
We consider a L\'evy process, for $t\geq 0$,
\begin{align}\label{eq:lev-example}
L_t(x) &= t a \mathbf{1}_{[0,y_1] \times \ldots\times {[0,y_d]}}(x)\nonumber\\
& \qquad + \int_0
^t\int_{\mathbb{R}^d}\int_0^\infty u\mathbf{1}_{[0,y_1] \times \ldots\times {[0,y_d]}}(x-y) N(\D t, \D y, \D u)\,,
\end{align}
where $a \in \mathbb{R}_+$, $\mathbf{1}_{A}$ is the indicator function of $A$ and $y_1, \ldots, y_d \in \mathbb{R}^d$.

Let $\phi \in L^1(\mathbb{R}^d)$, with $\phi \geq 0$. Then we can also define the L\'evy process by replacing the indicator function in \eqref{eq:lev-example} by such a $\phi$. Specifically 
\begin{align}\label{eq:levy-example1}
L_t(x) = t a \phi(x)+ \int_0
^t\int_{\mathbb{R}^d}\int_0^\infty u\phi(x-y) N(\D t, \D y, \D u)\,,
\end{align}

It follows that the Laplace transform of $(L_t)_{t\geq 0}$ is given, for all $g \in L^\infty(\mathbb{R}^d)$ by 
$$\mathbb{E}[\exp(-\langle g, L_t\rangle)] = \exp\left\{-t\left(\langle g, a \phi\rangle-\int_{\mathbb{R}^d}\int_0^\infty  (1- \e^{\langle g, u \phi(\cdot -y)\rangle})\, \pi(du)\,\lambda(\D y)\right)\right\}\,.$$
The L\'evy processes in \eqref{eq:levy-example1} and in \eqref{eq:lev-example} are $L^1_+(\mathbb{R}^d)$-valued L\'evy subordinators. Notice that  $\phi$ in \eqref{eq:levy-example1} represents the spatial distribution of the jumps and $u$ the intensity of the jumps. 
 \subsubsection{A L\'evy subordinator in $M_+(\mathbb{R}^d)$ with finite-dimensional noise}\label{example3} 
 Let $N(\D s, \D x, \D u)$ be a Poisson random measure on $\text{Bor}([0,\infty))\otimes \text{Bor}(\mathbb{R}^d)\otimes\text{Bor}((0,\infty))$ with intensity $\D s \otimes \lambda(\D x) \otimes \pi(\D u)$, where 
$\lambda(\D x)$ and $\pi(\D u)$  are as in Example \ref{example2}.

We consider a L\'evy process given by
\begin{equation}\label{eq:levy-example}
L_t = \int_0
^t\int_{\mathbb{R}^d}\int_0^\infty u\delta_x N(\D t, \D x, \D u)\,, \qquad t\geq 0\,.
\end{equation}
where $\delta_x$ is the Dirac measure on $\mathbb{R}^d$.
It follows that the Laplace transform of $(L_t)_{t\geq 0}$ is given, for all $g \in C_{0, +}(\mathbb{R}^d)$ by 
$$\mathbb{E}[\exp(-\langle g, L_t\rangle)] = \exp\left\{t\left(\int_{\mathbb{R}^d}\int_0^\infty  (1- \e^{-u g(x)})\, \pi(du)\,\lambda(\D x)\right)\right\}\,.$$
The L\'evy process in \eqref{eq:levy-example} is an $M_+(\mathbb{R}^d)$-valued L\'evy subordinator with finite-dimensional noise.
\subsubsection{An $M_+(E)$-valued Poisson process with finite-dimensional noise}\label{Levy-probability-measure} Let $\pi_0 =\delta_{z_0}$, for $z_0 \in \mathbb{R}_+$. Let $\varrho \in \mathbb{R}_+$ and $\Phi\colon \mathbb{R}_+\rightarrow M_+(E)$ given by $\Phi(z) = \mu z$, for $\mu \in M_+(E)$. 
Let $N(\D s, \D\nu)$ be a Poisson random measure on $\text{Bor}([0,\infty)) \otimes \text{Bor}(M_+(E))$ with intensity measure $\D s \otimes \pi(\D \nu)$, where $\pi(\D \nu)= \varrho \pi_0(\Phi^{-1}(\D \nu))$ is a L\'evy measure on $M_+(E)$. 
Let 
$$L_t =  \int_0^t \int_{M_+(E)} \nu N(\D s, \D\nu), \qquad t\geq 0\,.$$
Then $(L_t)_{t\geq 0}$ is another example of an $M_+(E)$-valued L\'evy subordinator with Laplace transform, for $g\in C_{0,+}(E)$, $t\geq 0$, 
$$\mathbb{E}[\e^{-\langle g, L_t\rangle}] = \exp\left(-t \varrho (1-\e^{-z_0\langle g, \mu\rangle})\right)\,. $$
I.e., this is an $M_+(E)$-valued Poisson process with jump intensity $\varrho$ and jumps fixed to be of size $\mu z_0$, for $z_0 \in \mathbb{R}_+$.

\subsection{Examples of L\'evy-Driven Measure-Valued CARMA($p,q$) Processes}\label{sec:examples}
\subsubsection{CAR($p$) in $L^{1}_{+}(E)$}
Let $p \in \mathbb{N}$. We define the projection on the $i$th coordinate as $\cP_{i}\colon L^{1}_{+}(E)^p\rightarrow
L^1_+(E)$, i.e., $\cP_i\mathbf{\bm \alpha}=\nu^{i}$, for $\bm \alpha\in L^1_+(E)^p$ and $i=1,\ldots,p$. Let $(\bm X_t)_{t\geq 0}$ be as in \eqref{eq:CARMA-banach}.
An {\it $L^1_+(E)$-valued continuous-time autoregressive process with parameter $p$} (in short referred to as CAR$(p)$) process is given by 
$$\cP_1 \bm X_t, \qquad t\geq 0\,.$$
The processes CAR$(p)$ in $L^1_+(E)$ constitute a subclass of the CARMA$(p,q)$ processes in $L^1_+(E)$.
\subsubsection{$\mathbb{R}$-valued L\'evy-driven CARMA$(p,q)$ process}

 In the matrix operator $\mathcal A_p$ in \eqref{eq:Ap-Banach}, let $A_i=-a_i\mathbb I$, $i=1,...,p$, where $a_i$ are positive real numbers and $\mathbb I$ is the identity operator on $L^1(\mathbb R^d)$. Then, $\mathcal A_p$ is a bounded linear operator and for $(L_t)_{t\geq 0}$ being a L\'evy process in $L^1(\mathbb R^d)$, it holds that $L^g_t:=\langle g,L_t\rangle$ is a L\'evy process on $\mathbb R$ with a L\'evy measure $\ell(g^{-1}(\cdot))$, for $g\in L^\infty(\mathbb R^d)$. Take now a Borel set $D\subset\mathbb R^d$, and define $\mathbf{g}=({\bf 1}_D,...,{\bf 1}_D)$. From the analytic weak solution, we find that (denoting $\mathbf{e}_p$ the standard $p$th basis vector in $\mathbb R^p$)
$$
\int_{D}\D \mathbf{X}_t(x )\, \D x=A_p\int_{D} \mathbf{X}_t(x )\, \D t\,\D x+ \int_D\mathbf{e}_p \,  \D L_t(x)\, \D x\,,
$$
for $A_p$ being the classical CARMA-matrix, i.e.,
\begin{align}
\label{carma-matrix-classical}
  A_{p}\df
  \begin{bmatrix}
    0 & 1 & 0 & \ldots & 0 \\
    0 & 0 & 1 & \ddots & \vdots \\
    \vdots & \vdots & \ddots   & \ddots & 0 \\
    0 & \ldots & \ldots & 0 & 1     \\
    -a_{p} & -a_{p-1} & \ldots & \ldots & -a_{1}
  \end{bmatrix}\,. 
\end{align}
Letting the operator $\cC_q\colon L^1(E)^p \to L^1(E)$ in \eqref{eq:Cq-Banach} be given by the identity operators $\mathbb I$ scaled by reals $c_i$, $i=0, \ldots,p-1$, i.e., $\cC_q = (c_0\mathbb I, c_1\mathbb I,\ldots, c_{p-1}\mathbb I)$, where  $c_q\neq 0$ and $c_j=0$, for $q+1 \leq j\leq p-1$. Let ${\bf c}_q = (c_0,\ldots, c_{p-1})$.
We find that $Y_t=\langle \mathbf{c}_q, \int_D\mathbf{X}_t(x)\, \D x\rangle_{\mathbb{R}^p}$ is a classical L\'evy-driven CARMA($p,q$)-process. So, in the signed-measure case, we recover a classical CARMA by evaluating our measure-valued CARMA in a set $D$. To ensure positivity, we need to assume that the measure-valued L\'evy process $(L_t)_{t\geq 0}$ is a subordinator on the cone of measures. Moreover we need to assume conditions on the $a_i$'s that ensure the positivity of the matrix $A_p$. We refer to \cite{BK23} for more on this.

\subsubsection{CARMA(p,q) processes in $L^1_+(E)$ and multi-parameter CARMA random fields}\label{sec:random-filed} 

 We introduce a class of multi-parameter CARMA processes from evaluating the measure-valued CARMA process on indicator functions. The obtained multi-parameter process is contrasted with the class of CARMA random fields proposed and analysed in \cite{pham2020levy}. Let $\cA_p$, $\cB_p$ be respectively as in \eqref{eq:Ap-Banach}, \eqref{eq:Bp-Banach}, and $(L_t)_{t\geq 0}$ be as in \eqref{eq:levy-example1}.

Let $(\bm X_t)_{t\geq 0}$ be given by
\begin{align}\label{eq:x-MRd}
d{\bm X}_t 
=\mathcal{A}_p \bm{X}_t\, dt  +  \mathcal{E}_p dL_t\,\quad t\geq 0\, \qquad \bm X_0=\bm \alpha \in L^1_+(\mathbb{R}^d)^p\,.
\end{align}

Let $(\cS_t)_{t\geq 0}\colon L^1_+(\mathbb{R}^d)^p \rightarrow L^1_+(\mathbb{R}^d)^p$ be the semigroup generated by $\cA_p$. Then solving \eqref{eq:x-MRd}, yields 
\begin{equation}
    \bm X_t = \cS_t \bm \alpha + \int_0^t\cS_{t-s}\cB_p\, \D L_s\,, \qquad t\geq 0\,.
\end{equation}

Take $g(x_1, \ldots, x_d)= \mathbf{1}_{[0,t_1]}(x_1)\mathbf{1}_{[0,t_2]}(x_2)\ldots \mathbf{1}_{[0,t_d]}(x_d)$, for $t_d \leq\ldots\leq t_2\leq t_1$.
Then it holds 
\begin{align}\label{eq:OU-mean-field}
\mathbb{X}(t, t_1, \ldots, t_d) &:= \int_0^{t_1}\ldots \int_0^{t_d} \bm X_t(\D\bm y)\nonumber\\ 
&= \int_0^{t_1}\ldots\int_0^{t_d}\cS_t \bm{\alpha}(\D \bm y)
+ \int_0^t\int_0^{t_1}\ldots\int_0^{t_d}\cS_{t-s}\mathcal{E}_p \,\D L_u(\D\bm y)\,.
\end{align}
Let $\bm c_q= (c_0, \ldots, c_{p-1}) \in \mathbb{R}^p$, where $c_j$, satisfy $c_q \neq 0$ and $c_j =0$, for $q< j\leq p$.
We define {\it a random field 
CARMA$(p,q)$} as 
$$\mathbb{Y}(t, t_1,\ldots, t_d) =\langle \bm c_q ,\mathbb{X}(t,t_1, \ldots, t_d)\rangle_{\mathbb{R}^p}.$$ 

Our CARMA random field process is different from the one introduced in \cite{pham2020levy}. Indeed our semigroup $(\cS_t)_{t\geq 0}$ is operating on measures while, in the notation of \cite[equation 3.3]{pham2020levy}, $\e^{A_i(t)}$, $i=1,\ldots, d$, are matrices. In addition, integration on the right-hand side of \eqref{eq:OU-mean-field} is considered with respect to a L\'evy process in $L^1_+(\mathbb{R}^d)$, while integration in \cite[equation 3.3]{pham2020levy} is with respect to a L\'evy basis. 
Note that our definition of the CARMA random field is not limited to Lévy processes of the form \eqref{eq:levy-example1}; it can be similarly defined for any Lévy process in $L^1_+(\mathbb{R}^d)$. 

\subsubsection{Measure valued CARMA processes driven by a finite-dimensional noise}
Let $\mathcal{A}_p, \mathcal{E}_p, \mathcal{C}_q$ be respectively as in \eqref{eq:Ap-Banach}, \eqref{eq:Bp-Banach} and~\eqref{eq:Cq-Banach}, where we specify the Banach space $B$ to be the space of measures $M(\mathbb{R}^d)$. Assume $\mathcal{A}_p \colon D(\cA) \subset M_+(\mathbb{R}^d)^p \rightarrow M_+(\mathbb{R}^d)^p$ is the generator of a strongly continuous, quasi-positive operator semigroup $(\mathcal{S}_t)_{t \geq 0}$, $C_j \in \pi(M_+(\mathbb{R}^d))$, for all $j\in \{0,\ldots, q\}$, and $\sE(M_+(\mathbb{R}^d)) \subseteq M_+(\mathbb{R}^d)$. Let $(L_t)_{t \geq 0}$ be a $M_+(\mathbb{R}^d)$-valued Lévy process in $M(\mathbb{R}^d)$ as specified in Example \ref{example3} or in Example \ref{Levy-probability-measure}. 
Let $(\bm X_t)_{t\geq 0}$ and $(Y_t)_{t\geq 0}$ be given by 
 \begin{equation*}
\D \mathbf{X}_t = \mathcal{A}_p \mathbf{X}_t \, \D t + \mathcal{E}_p \, \D L_t, \quad Y_t = \mathcal{C}_q \mathbf{X}_t,
\end{equation*}
 where $\mathbf{X}_0 = \bm \nu\in M_+(\mathbb{R}^d)^p$. The existence of $(\bm X_t)_{t\geq 0}$ is guaranteed by Proposition \ref{prop:OU-mild-sol} and Remark \ref{rem:finite-dim-noise}. The process $(Y_t)_{t\geq 0}$ is {\it a pure-jump measure-valued CARMA$(p,q)$} process with parameter set $(\mathcal{A}_p, \mathcal{E}_p, \mathcal{C}_q, L)$ driven by finite-dimensional noise.
 
 We can further specify $\cA_p$ and $\cE_p$ to be of convolution type, similar to the $L^1(\mathbb{R}^d)$ case, where {\it a convolution operator} in the space of measures is defined as $T_\mu \colon M_+(\mathbb{R}^d) \rightarrow M_+(\mathbb{R}^d)$,
for some fixed measure $\mu \in M_+(\mathbb{R}^d)$, where $$(T_\mu \nu)(B) = \mu * \nu(B) = \int_{\mathbb{R}^d}\nu(B- x)\,\mu(\D x)\,,$$
for all $B \in \text{Bor}(\mathbb{R}^d)$. Notice that $T_\mu$ is a bounded operator as $\nu$ and $\mu$ are assumed to be finite measures. 
\subsubsection{Measure-valued ambit fields}
Our CARMA processes can be viewed as a special class of ambit fields. 

Let $\mathcal{A}_p, \mathcal{E}_p, \mathcal{C}_q$ be respectively as in \eqref{eq:Ap-Banach}, \eqref{eq:Bp-Banach} and~\eqref{eq:Cq-Banach}, where we specify the Banach space $B$ to be $(L^1(\mathbb{R}^d), \|\cdot\|_{L^1})$. Let $(\cS_{t})_{t\geq 0}$ be the semigroup generated by $\cA_p$ and $(L_t)_{t\geq 0}$ be an $L^1_+(E)$-valued L\'evy process with characteristic triplet $(\gamma, 0, \ell)$, where $\gamma \in L^1_+(E)$ and $\ell$ is the Lévy measure concentrated on $L^1_+(E)$ satisfying \eqref{eq:int-measure} and Theorem~\ref{pettis} with the cone $K$ being  $L^1_+(E)$. Let $N(\D s, \D \nu)$ denote the Poisson random measure associated with the jumps of $(L_t)_{t \geq 0}$ and $\tilde{N}(\D s, \D \alpha)$ be the compensated Poisson random measure. From \eqref{eq:variation-of-constant} and the fact that $Y_t =  \cC_q \bm X_t$, it holds for $\bm \alpha \in L^1_+(E)^p$,
\begin{align*}
Y_t&=\cC_q\cS_t \bm \alpha+\int_0^t\cC_q\cS_{t-s}\cB_p\gamma \D s+\int_0^t\int_{\set{\alpha\in L_+^{1,\circ}(E)\colon 0<\norm{\alpha}_{L^1}\leq 1}}\cC_q\cS_{t-s}\cB_p \alpha\tilde{N}(\D s,\D \alpha)\nonumber\\
    &\quad +\int_0^t\int_{\set{\alpha \in L_+^{1,\circ}(E)\colon \norm{\alpha}_{L^1}> 1}} \cC_q\cS_{t-s}\mathcal{E}_p\alpha N(\D s,\D \alpha).
\end{align*}
Notice that the Poisson and the compensated Poisson random measures $N$ and $\tilde{N}$ qualify as L\'evy bases on $[0,T] \times L^1_+(E)$ according to \cite[Definition 25]{BNBV18}.
Therefore, as the stochastic integrals are defined in the weak Pettis sense (see the proof of Proposition \ref{prop:OU-mild-sol}), we have for $g \in L^\infty(E)$,
\begin{align}\label{eq:ambit-field}
\langle g, Y_t\rangle&=\langle g, \cC_q\cS_t\bm \alpha\rangle +\int_0^t\langle g, \cC_q\cS_{t-s}\cB_p\gamma \rangle\,\D s \nonumber\\
&\qquad \qquad +\int_{0}^{t} \int_{\set{\alpha\in L_+^{1,\circ}(E)\colon 0<\norm{\alpha}_{L^1}\leq 1}}\langle g,\mathcal C_q S_{t-s}\cB_p\alpha\rangle\tilde{N}(\D s,\D \alpha) \nonumber\\
&\qquad\qquad+\int_0^t\int_{\set{\alpha\in L_+^{1,\circ}(E)\colon \norm{\alpha}_{L^1}> 1}}\langle g,\mathcal C_q\cS_{t-s}\cB_p\alpha\rangle N(\D s,\D \alpha) 
\end{align}
is a real-valued {\it ambit field}  without any volatility modulation for every $g \in L^\infty(E)$ and $(t,A) \mapsto \int_A Y_t \, \D x$ is {\it a measure-valued ambit field} (see \cite[Definition 44]{BNBV18} for the definition of ambit fields). When we assume $(\cS_t)_{t\geq 0}$ is a quasi-positive operator, $C_j \in \pi(L^1_+(E))$, for all $j\in \{0,\ldots, q\}$, and $\sE(L^1_+(E)) \subseteq L^1_+(E)$, then $(Y_t)_{t\geq 0}$ is an $L^1_+(\mathbb{R}^d)$-valued process.

Based on the above discussions, we provide here an outlook to a generalisation of measure-valued CARMA processes to what we call {\it measure-valued ambit processes}. First, let us extend the noise in \eqref{eq:ambit-field} and consider martingale valued measures $U$ on $[0,T] \times L^1_+(E)$ as introduced in \cite[Section 2]{RvG09}. Here we assume that $U$ satisfies all conditions (a)--(f) of Section 2 in \cite{RvG09}. Notice that a real-valued L\'evy basis on $[0,T]\times L^1_+(E)$ is an infinitely-divisible martingale valued measure (see \cite[Definition 25]{BNBV18}). 
Let $\rho \colon \mathbb{R}_+ \times L^1_+(E) \to \mathbb{R}$ be the square mean measure associated with $U$, the existence of which is guaranteed by condition (f) in Section 2 of \cite{RvG09}.
Furthermore let $\mathcal G:[0,T]^2\times L^1_+(E)\mapsto L^1_+(E)$ be such that 
$$\int_{[0,T] \times L^1_+(E)} \langle 1, \cG(t,s,\alpha) \rangle^2 \, \rho(\D s, \D \alpha) <\infty\,.$$
Then according to \cite[Theorem 3.6]{RvG09}, there exists an $L^1_+(E)$-valued process $(Y_t)_{t\geq 0}$ such that for all $g \in L^\infty(E)$, we have 
\begin{align*}
\langle g, Y_t\rangle  = \int_{[0,t]\times L^1_+(E)} \langle g, \cG(t,s,\alpha)\rangle\, U(\D s , \D \alpha)\,, \qquad \text{a.s.}\,
\end{align*}
Then the process $(t,A) \mapsto \int_A Y_t(x) \, \D x$ is a measure-valued ambit field.

\section{Calculating expectation functionals}\label{sec:appication-expectation}

In many applications, one is interested in computing expectation functionals,
i.e., nonlinear mappings of the process in question. For example, one can think
of pricing an option on flow forwards $F(\tau, \tau_1, \tau_2)$, $\tau\leq \tau_1<\tau_2$, as introduced in \eqref{eq:flow-forward}. This entails in computing a (possibly
risk-adjusted) expected value of the payoff of the option, given as a function 
\begin{align}\label{eq:future}
\Upsilon(F(\tau,\tau_1,\tau_2))
\end{align}
at some exercise time $\tau$, where $\Upsilon\colon \mathbb{R}\to \mathbb{R}$ is the payoff function. 
Another example is computing the expected income from power production from a wind or solar power plant, given by some (possibly non-linear) map of the wind or solar irradiation field over an area, over a span of time, where the wind or irradiation field is modelled by measured-valued processes. We refer to the discussion in Section \ref{sec:intro} for further motivations on the relevance of such expectation functionals.   

To this end, consider a CARMA process $(Y_t)_{t\geq 0}$ in a separable Banach space $B$ as introduced in Definition \ref{def:carma-process1}. We assume that the parameters satisfy the conditions outlined in Proposition \ref{prop:pos-necessary-car} to ensure that the CARMA process remains in the cone $K$. For an element $h\in K^*$ and a mapping $\Upsilon:\mathbb R\rightarrow\mathbb R$, we derive in the next proposition, an expression for the expectation functional
\begin{equation}\label{eq:expectation-functional}
\Pi_t:=\mathbb E[\Upsilon(\langle h,Y_{\tau}\rangle)\,\vert\,\mathcal F_t]\,,
\end{equation}
for $0\leq t\leq\tau$. This proposition is applicable when $B$ is the space of absolutely continuous measures, as relevant to our applications. Additionally, it can be applied to 
$M_+(E)$ when considering finite-dimensional noise, as discussed in Remark \ref{rem:finite-dim-noise}.  

\begin{proposition}\label{prop:Fourier-transform} Let $\cA_p$, $\cB_p$, $\cC_q$, and $(L_t)_{t\geq 0}$ be as in Definition \ref{def:carma-process1}.  
Assume $\cA_p$ to be quasi-positive with respect to $K^p$, $C_j \in \pi(K)$, for all $j\in \{0,\ldots, q\}$, and $\sE(K) \subseteq K$. Let $(L_t)_{t \geq 0}$ be a $K$-valued Lévy process in $B$ with characteristic triplet $(\gamma, 0, \ell)$, where $\gamma \in K$ and $\ell$ is the Lévy measure concentrated on $K$ satisfying Theorem~\ref{pettis}.

Assume $\Upsilon$ can be expressed as 
\begin{align}\label{eq:Fourier-representation}
\Upsilon(x) = \int_\MR \mathrm{e}^{(a+\mathrm{i}y)x} \hat{\Upsilon}(y) \, \D y\,,
\end{align}
for a function $\hat{\Upsilon}\in L^1(\MR)$ and some $a \in \MR$ such that $\mathbb{E}[\exp\{a\langle h, Y_\tau\rangle\}]<\infty]$, for all $h \in K^{*}$. Then we have for $h \in K^{*}$,
\begin{align*}
\Pi_t &= \int_\MR \exp\left\{-\langle \cS_t^*\mathcal{C}_q^*(a+iy)h, {\bf X}_t\rangle - \int_0^{\tau -t}\langle \mathcal{E}^*_p\cS_{\tau-s}^*\mathcal{C}_q^*(a+iy)h, \gamma\rangle\, \D s\right\}\\
&\qquad \exp\left\{\int_0^{\tau -t}\int_{K}(\e^{-\langle \mathcal{E}^*_p\cS_{\tau -s}^*\mathcal{C}^*_q (a+iy)h, \nu\rangle}-1)\,\ell(\D \nu)\, ds\right\}\hat{\Upsilon}(y)\, \D y\,.
\end{align*}
\end{proposition}
\begin{proof}
From the representation of $\Upsilon$ we have
$$
\Upsilon(\langle h,Y_{\tau}\rangle)=\int_{\MR}\exp((a+iy)\langle h,Y_{\tau}\rangle))\hat{\Upsilon}(y)\,\D y\,.
$$
From the assumption of the proposition, we find appealing to Fubini's Theorem that 
\begin{align*}
    \mathbb E[\Upsilon(\langle h,Y_{\tau}\rangle)\mid \cF_t]&=\int_{\MR}\mathbb E[\exp((a+iy)\langle h,Y_{\tau}\rangle)\,\vert\,\mathcal F_t]\hat{\Upsilon}(y) \,\D y \\
    &=\int_{\MR}\mathbb E[\exp(-\langle -(a+iy)\mathcal C_q^*h,{\bf X}_{\tau}\rangle_p)\,\vert\,\mathcal F_t]\hat{\Upsilon}(y)\, \D y.
\end{align*}
The result follows from Proposition \ref{prop:transition-semigroup}.
\end{proof}

Note that if $x\mapsto\Upsilon(x)\exp(-ax)$ is integrable on $\MR$ with a Fourier transform also being integrable, then, by the Fourier inversion formula, 
$$
\Upsilon(x)=e^{ax}e^{-ax}\Upsilon(x)=e^{ax}\frac1{2\pi}\int_{\MR}e^{ixy}\widehat{\Upsilon}_a(y)\, \D y\,,
$$
where $\widehat{\Upsilon}_a$ is the Fourier transform of $x\mapsto e^{-ax}\Upsilon(x)$. Hence, $\Upsilon$ has a representation as in \eqref{eq:Fourier-representation}, where
$\hat{\Upsilon}(y)=\widehat{\Upsilon}_a(y)/2\pi$.

\begin{remark}
When pricing financial derivatives, such as options written on power futures, one usually resorts to the arbitrage-free pricing theory and asks for a risk-neutral probability which turns the forward price dynamics into a martingale. The price is given as the expectation operator of the payoff under this risk-neutral probability. However, this approach rests on the fact that the futures are liquidly tradeable. For example, in many organized markets, the futures are not necessarily very liquid, and entering such a contract may very well lock you in the position. In particular, having in mind OTC-contracts such as power purchase agreements or other production and weather-linked derivatives, the  situation is often that this is a position that might be hard to reverse. As such, the derivative should be priced under a risk-adjusted measure, which is not related to any replicating strategy and martingale condition of the underlying, but measuring risk-compensation inherit in the contract. Thus, if modelling the underlying dynamics by a measure-valued CARMA processes, one may ask for a class of equivalent measures that can be used for this risk-adjustment. We discuss here the Esscher transform, adopted from insurance mathematics, see, e.g., \cite{escher1932probability, gerber1993option, kallsen2002cumulant}, which turns out to be a measure change preserving the measure-valued CARMA structure but modifying the parameters in the model. 
\end{remark}

\subsection{Change of measure}\label{sec:change-of-measure}
Let $\cB\colon K \rightarrow K^p$. Let $\ell$ be a L\'evy measure with support in $K$.
Define 
\begin{equation}\label{eq:measure-LB}
  \ell^{\cB}(A) = \ell\{\nu \in K \mid \cB(\nu) \in A\}, \quad  A \subset \mbox{Bor}(K^p)\,.
\end{equation}
Let $(\mathbfcal{L}_t)_{t\geq 0}$ be a  L\'evy process in $K$ with characteristics $(0,0, \ell^\cB)$.  
Hence for $\bm{\theta}\in (B^*)^p$ such that 
\begin{align}\label{eq:exp-moments}
\int _{K}\mathrm{e}^{\langle \bm{\theta}, \bm{\nu}\rangle_p}\ell^{\mathcal{E}}(d\bm{\nu})<\infty\,,
\end{align}
we introduce
\begin{align*}
    Z_t^{\bm{\theta}} = \exp\left\{ \langle \bm{\theta},\mathbfcal L_t\rangle_p- t \int_{K^p}(\mathrm{e}^{\langle \bm{\theta}, \bm{\nu}\rangle_p}-1)\, 
\ell^{\mathcal{E}}\,(\D \bm \nu)\right\}\,.
\end{align*}
Denote by $N^\mathcal{E} (\D s, \D \bm \nu)$ the Poisson random measure on $\mbox{Bor}([0,\infty]) \times \mbox{Bor}(K^p)$ with compensator $\D s\, \ell^{\mathcal{E}}(\D\bm\nu)$. According to  Theorem \ref{thm:cone-subordinator}, it holds 
\begin{align}\label{eq:representation-L}
\langle \bm{g}, \bm{\mathcal{L}}_t \rangle = \int_0^t\int_{K^p} \langle \bm g, \bm\nu\rangle  N^\mathcal{E}(\D s, \D\bm \nu), \qquad \forall \bm g \in (K^*)^p\,.
\end{align}

In the following two lemmas we introduce an Esscher type transform and we compute the dynamics of $(\bm{\mathcal{L}}_t)_{t\geq 0}$ under the new measure.
\begin{lemma}
   Let $N^\cB$ be as described in \eqref{eq:representation-L}. Let $\bm\theta$ satisfy \eqref{eq:exp-moments}.

    Then there exists a probability measure $\MQ_T^{\bm \theta}$ on $\mathcal{F}_T$ that is absolutely continuous with respect to $\MP$. If we denote by $\MP_t$, respectively $\MQ_t^{\bm \theta}$, the restriction of $\MP$, respectively $\MQ_T^{\bm \theta}$ to $\mathcal{F}_t$, then $$\frac{d \mathbb{Q}^{\bm \theta}}{d\mathbb{P}}\big\vert_{{\mathcal{F}_t}} = Z_t^{\bm \theta}$$
defines a family of densities for every $t\leq T$.
\end{lemma}
\begin{proof}
Observe that 
\begin{align}\label{eq:martingale-measure}
\ln\mathbb E[\exp(\langle\bm \theta,\mathbfcal L_t\rangle_p)]=t\int_{K^p} (\exp(\langle \bm \theta,\bm \nu\rangle_p)-1) \,\ell^\mathcal{E}(\D \bm \nu)\,.
\end{align}
By the latter and the independent increments of $(\bm \cL_t)_{t\geq 0}$,
we deduce that $(Z_t^{\bm \theta})_{t\geq 0}$ is a martingale with $\mathbb{E}[Z_t^{\bm \theta}] =1$. Hence the statement of the lemma follows. 
\end{proof}
We show in the following lemma that under $\mathbb Q^{\bm \theta}_t$, $\mathbfcal L_t$, $t\geq 0$, is a pure-jump L\'evy process without drift with the L\'evy measure 
$$
e^{\langle\bm \theta,\bm \nu\rangle_p}\, \ell^\cB(\D \bm \nu)
$$
supported on $K^p$, i.e., an exponential tilting of the original L\'evy measure $\ell^\cB$ in accordance with the "classical" Esscher transform. 

\begin{lemma}\label{lem:lnew-dynamoics}
Let $(\mathbfcal{L}_t)_{t\geq 0}$ be a  L\'evy process in $K^p$ with characteristics $(0,0, \ell^\cB)$. Assume there exists $\bm \theta \in (B^*)^p$ satisfying \eqref{eq:exp-moments}. It holds that 
$(\mathbfcal{L}_t)_{t\geq 0}$ remains a pure-jump L\'evy process with characteristics $(0,0, \ell_{\bm \theta}^\cB)$ under $\mathbb{Q}^{\bm \theta}_t$, for all $t\leq T$, where 
$$\ell_\theta^\cB(\D \bm \nu) = \mathrm{e}^{\langle \bm \theta, \bm \nu\rangle_p}\, \ell^\mathcal{E}(\D \bm \nu)\,.$$ 
\end{lemma}

\begin{proof}
Let $\ME_{\bm \theta}$ denote the expectation under the measure $\mathbb{Q}^{\bm \theta}_T$. Computing the characteristic exponent of $\mathbfcal L_T$ with respect to $\mathbb Q^{\bm \theta}_T$, yields, for any $\bm g\in (B^*)^p$,
\begin{align*}
    &\log\mathbb E_\theta\left[\exp(\mathrm{i}\langle \bm g,\mathbfcal L_t\rangle_p)\right]\\
    &\quad =\log\mathbb E\left[\exp(\mathrm{i}\langle \bm g,\mathbfcal L_t\rangle_p)\frac{d\mathbb Q^{\bm \theta}_T}{d\mathbb P}_{\big\vert_{\mathcal F_t}}\right] \\
    &\qquad=\log\mathbb E\left[\exp(\langle\mathrm{i} \bm g+\bm \theta,\mathbfcal L_t\rangle_p)\right]- \left[t\int_{K^{p}}(e^{\langle\bm \theta,\bm \nu\rangle_p}-1)\,\ell^\mathcal{E}(d\bm \nu)\right] \\
    &\qquad =t\int_{K^p}\left(e^{\mathrm{i}\langle \bm g, \bm \nu\rangle_p}-1\right) e^{\langle\bm \theta, \bm \nu\rangle_p}\, \ell^\mathcal{E}(\D \bm\nu)
\end{align*}
and the statement holds.
\end{proof}
In the following proposition we derive the generating function of our measure-valued state space process under the measure $\mathbb{Q}^{\bm \theta}_T$.
\begin{proposition}
    Let $(\bm X_t)_{t\geq 0}$ be a state space process as in Definition \ref{def:carma-process1} with parameters $\cA_p$, $\cB_p$, $\cC_q$, and $(L_t)_{t\geq 0}$ as described in Proposition \ref{prop:Fourier-transform}. Let the L\'evy measure $\ell^{\cB_p}$ on $K^p$ be as defined in \eqref{eq:measure-LB}, for $\cB =\cB_p$ and $\ell$ being the L\'evy measure of the driver $(L_t)_{t\geq 0}$ of the measure-valued OU $(\bm X_t)_{t\geq 0}$. Assume there exists $\bm \theta \in (B^*)^p$ satisfying \eqref{eq:exp-moments}. It holds for $\bm g\in (K^*)^p$
    \begin{align*}
    &\ME_\theta[\mathrm{e}^{-\langle \bm g, \bm X_t\rangle_p}]\\
    &\quad = \e^{-\langle \bm g, \cS_t\bm\nu\rangle_p}\\
    &\qquad \times  \exp\left\{-t\left(\langle \cS^*_{t-s} \bm g, \bm \gamma_0^{\bm \theta} \rangle_p + \int_{K^p}\left(1-\e^{-\langle  \cS^*_{t-s} \bm g, \bm \nu\rangle_p}\right) \e^{\langle \bm \theta, \bm \nu\rangle_p}\, \ell^{\cB_p}(\D \bm \nu) \right)\right\} \,, 
    \end{align*}
    where $\bm \gamma_0^{\bm \theta} = \cB_p \gamma + \int_{\{\|\bm \nu\|_p \leq 1\}}  \bm \nu \e^{\langle \bm \theta, \bm \nu \rangle_p}\, \, \ell^{\cB_p} (\D \bm \nu)$.
\end{proposition}
\begin{proof} 
Denote by $N_{\bm \theta}^{\cB_p}(\D s, \D \bm \nu)$ the Poisson random measure with intensity $\D s \, \ell^\cB_{\bm \theta}(\D \bm \nu)$ and by $\tilde{N}_{\bm \theta}^{\cB_p}(\D s, \D \bm \nu)$ the compensated Poisson random measure.  Observe that from Lemma \ref{lem:lnew-dynamoics}, the dynamics of $\bm X_t$ under $\mathbb{Q}_t^{\bm \theta}$, $t \geq 0$, 
for any test function $\mathbf{g} \in D(\mathcal{A}^*_p) \cap (B^*)^{p}$, and $t \geq 0$, is given by
\begin{align*}
  \langle \mathbf{g}, \mathbf{X}_t \rangle_p &= \langle \mathbf{g}, \mathbf{X}_0 \rangle_p +t\langle \mathbf{g}, \cB_p\gamma \rangle_p+\int_0^t \langle \mathcal{A}^*_p \mathbf{g}, \mathbf{X}_s \rangle_p \D s \nonumber\\ &\quad+ t\int_{\set{ \norm{\bm \nu}_p< 1}} \langle \mathbf{g}, \bm \nu\rangle_p (\e^{\langle \bm \theta , \bm \nu\rangle_p} - 1) \, \ell^{\cB_p}(\bm \D \nu)\nonumber\\ &\quad+ \int_0^t\int_{\set{ \norm{\bm \nu}_p< 1}} \langle \mathbf{g}, \bm \nu \rangle_p \, \tilde{N}_{
  \bm \theta}^{\cB_p}(\D s,\D \bm \nu) \nonumber\\
&\quad   + \int_0^t\int_{\set{ \norm{\bm \nu}_p\geq 1}} \langle \mathbf{g}, \bm \nu \rangle_p \, N_{\bm \theta}^{\cB_p}(\D s,\D \bm \nu)\,.
\end{align*}
The expression for the Laplace transform follows from Proposition \ref{prop:transition-semigroup}.
\end{proof}

Note that $\bm \theta\in (B^*)^p$, meaning that $\bm \theta$ is a $p$-vector of functions that might take negative values and
therefore we may have $\langle\bm \theta,\bm \nu\rangle_p<0$. For example, this opens up for flexibility in modelling both negative and positive risk premia in electricity flow forward markets (see e.g. \cite{BK08}.).

\bibliographystyle{acm}
\bibliography{reference_list}
\end{document}